\documentclass[10pt]{amsart}

\usepackage{amsmath,amsthm,amsfonts, bbm, amssymb, graphicx,color, accents}

\newcommand{\ii}{\mathbbm{i}}

\newcommand{\C}{\mathbb{C}}
\newcommand{\CP}{\mathbb{CP}}

\newcommand{\Heis}{\mathbb{H}}
\newcommand{\Sieg}{\mathcal{S}}
\newcommand{\R}{\mathbb{R}}
\newcommand{\Z}{\mathbb{Z}}
\newcommand{\Q}{\mathbb{Q}}
\newcommand{\N}{\mathbb{N}}
\newcommand{\J}{\mathbb{J}}

\newcommand{\interior}[1]{\accentset{\circ}{#1}}

\newcommand{\norm}[1]{\left\vert #1 \right \vert}	
\newcommand{\Norm}[1]{\left\Vert #1 \right \Vert}
\newcommand{\floor}[1]{\left\lfloor #1 \right\rfloor}		
\newcommand{\class}[1]{\left[ #1 \right]}

\renewcommand{\Re}{\operatorname{Re}}
\renewcommand{\Im}{\operatorname{Im}}

\newcommand{\Frac}{\mathbb K}

\newcommand{\rad}{\operatorname{rad}}
\newcommand{\diam}{\operatorname{diam}}

\def\[#1\]{\begin{align}#1\end{align}}
\def\(#1\){\begin{align*}#1\end{align*}}

\newtheorem{thm}{Theorem}
\numberwithin{thm}{section}

\newtheorem{lemma}[thm]{Lemma}
\newtheorem{cor}[thm]{Corollary}
\newtheorem{question}[thm]{Question}

\theoremstyle{definition}
\newtheorem{defi}[thm]{Definition}

\theoremstyle{definition}

\newtheorem{remark}[thm]{Remark}

\numberwithin{equation}{section}

\newcommand{\st}{\;:\;}

\title[Heisenberg Continued fractions]{Continued Fractions on the Heisenberg Group}
\author[A. Lukyanenko]{Anton Lukyanenko}
\address{
Department of Mathematics\\
University of Illinois at  Urbana-Champaign\\
1409 West Green Street\\
Urbana, IL 61801, USA}
\urladdr{http://lukyanenko.net}
\email{anton@lukyanenko.net}
\author[J. Vandehey]{Joseph Vandehey}
\address{
Department of Mathematics\\
University of Illinois at  Urbana-Champaign\\
1409 West Green Street\\
Urbana, IL 61801, USA}
\email{vandehe2@illinois.edu}
\thanks{The first author acknowledges support from the National Science Foundation grants DMS-0838434 and DMS-1107452.}


\begin{document}

\begin{abstract}We provide a generalization of continued fractions to the Heisenberg group. 
We prove an explicit estimate on the rate of convergence of the infinite continued fraction and several surprising analogs of classical formulas about continued fractions. We then discuss dynamical properties of the associated Gauss map, comparing them with base-$b$ expansions on the Heisenberg group and continued fractions on the complex plane.
\end{abstract}

\date{\today}

\maketitle

\section{Introduction}
A regular continued fraction (RCF) expansion represents an irrational number $x\in \R$ as
\[x = a_0 + \cfrac{1}{a_1+\cfrac{1}{a_2+\cdots}}, \qquad a_0 \in \mathbb{Z}, \quad a_i \in \mathbb{N}, i \ge 1.\]
The integers $CF(x):=\{a_0, \ldots,\}$ are the \emph{continued fraction digits} of $x$ (also called the \emph{partial quotients}). Regular continued fractions and their many variations have played an important part in Diophantine approximation, hyperbolic geometry, and the study of quadratic irrationals. 

Many higher-dimensional generalizations of continued fractions have been developed to extend this powerful theory, but these efforts have met with varying success.  In this paper, we develop a notion of continued fractions in the non-commutative setting of the Heisenberg group (in a sense, a \emph{complex} two-dimensional continued fraction). Surprisingly, we recover not only standard results of convergence (see Theorem \ref{thm:intro:convergence}), but also several simple, direct analogs of classical formulas for regular continued fractions---formulas which lack simple analogs for any other known multi-dimensional continued fraction.  This suggests that continued fractions are a reasonable and natural object of study on the Heisenberg group.

This paper provides the basic properties of Heisenberg continued fractions, and opens up the way for many new questions. In future papers, we intend to link our study to that of Diophantine approximation on the Heisenberg group (see \cite{MR1930990}) and the behavior of geodesics in complex hyperbolic space (see Remark \ref{rmk:complexhyperbolic} and \cite{MR810563}). Additional interesting questions include extending these results to similar spaces and their lattices (specifically, we expect our results to hold for all boundaries of hyperbolic rank-one symmetric spaces), a characterization of periodic continued fraction expansions, a careful analysis of the dynamical properties of the associated Gauss map, and a description of the dual space.

The setting for this paper will be the Heisenberg group $\Heis$, arguably the most natural non-commutative generalization of $\R$. Specifically, $\Heis$ is $\R^3$ with the modified group law (which we denote by $*$)
\[ (x,y,t)*(x',y',t') = (x+x', y+y', t+t'+2(xy'-yx')).\]
Note that in the first two coordinates one sees the usual addition of vectors, while the third coordinate incorporates an antisymmetric term. Note also that the group inverse $(x,y,t)^{-1}$ of an element $(x,y,t)\in \Heis$ is given by $(-x,-y,-t)$.

Let $\Heis(\Z)$ denote the set of points in $\Heis$ with all integer coordinates. These form a subgroup of $\Heis$, and we will think of them as the integers within $\Heis$. Likewise, we think of points with all rational coordinates, $\Heis(\Q)$, as rational points.

Given a generic point $h\in \Heis$ there is a unique nearest Heisenberg integer $[h]\in \Heis(\Z)$, with respect to the Heisenberg group's standard \emph{gauge metric}:
\[ &\Norm{(x,y,t)}=\sqrt[4]{(x^2+y^2)^2+t^2} &d(h,k)=\Norm{h^{-1}*k}.\]
Note that \emph{left} translations by elements of $\Heis$ are isometries. That is, $d(g*h, g*k)=d(h,k)$ for all $g,h,k\in\Heis$. In addition, one has an inversion operation (see \S \ref{sec:geometric}) $\iota: \Heis\backslash\{0\} \rightarrow \Heis\backslash\{0\}$ satisfying 
\( \Norm{\iota(h)} = \Norm{h}^{-1}.\)
Given a point $h\in H$, we may remove the integer part of $h$ via $[h]^{-1}*h$.

\begin{defi}The \emph{continued fraction digits} $CF(h) = \{\gamma_i\}$ and \emph{forward iterates} $\{h_i\}$ of a point $h\in \Heis$ are defined inductively by:
\(
&\gamma_0=[h] &&h_0=\gamma_0^{-1}*h,\\
&\gamma_{i+1}=[\iota(h_i)] &&h_{i+1}=\gamma_{i+1}^{-1}*\iota(h_i).
\)
\end{defi} 

Note that $\iota(0)$ is undefined. Thus, the process may terminate after finitely many steps. We will characterize points for which this happens in Theorem \ref{thm:finiterational} and, for the majority of the paper, focus our attention on points with infinitely many digits. We will also generally assume that $\gamma_0=0$ unless otherwise specified.

\begin{defi}
Let $\{\gamma_i\}$ be a sequence of elements of $\Heis(\Z)$. For a finite sequence, define the associated continued fraction,
\( \Frac \{\gamma_i\} = \Frac \{\gamma_i\}_{i=0}^n := \gamma_0 \iota \gamma_1 \iota \cdots \iota \gamma_n,\)
 supressing product notation and parentheses.  It is clear that if $CF(h)$ is finite, then $\mathbb K CF(h)=h$.  

For an infinite sequence, we write
\( \Frac \{\gamma_i\} =\Frac \{\gamma_i\}_{i=0}^\infty :=\lim_{n\rightarrow\infty} \Frac \{\gamma_i\}_{i=0}^n,\)
provided the limit exists.
\end{defi}

Our main result is to show that $\Frac$ and $CF$ define a valid notion of a continued fraction expansion for a point in $\Heis$. Namely, we prove
\begin{thm}
\label{thm:intro:convergence}The following properties hold:
\begin{enumerate}
\item Let $\{\gamma_i\}$ be a sequence of elements of $\Heis(\Z)$ satisfying  $\Norm{\gamma_i} \geq 3$ for each $i$. Then $\mathbb K \{\gamma_i\}$ exists and is unique regardless of whether $\{\gamma_i\}$ is finite or infinite (Theorem \ref{thm:Pringsheim}).
\item A point $h\in \Heis$ satisfies $h = \Frac\{\gamma_i\}_{i=0}^n$ for a finite sequence $\{\gamma_i\}$ of elements of $\Heis(\Z)$ if and only if $h \in \Heis(\Q)$ (Theorem \ref{thm:finiterational}).
\item Every point in $\Heis$ has a continued fraction expansion. That is, for all $h\in \Heis$, the limit $\mathbb K CF(h)$ is unique and equal to $h$ (Theorem \ref{thm:conv}).
\end{enumerate}
\end{thm}

Throughout \S \ref{sec:heiscf}, we  obtain variants of classical continued fraction results. We show a relationship between the denominator of a rational point and the length of its continued fraction expansion in Theorem \ref{thm:denominator}. We find a recursive formula for the approximants $\Frac \{\gamma_i\}_{i=1}^n$ in Theorem \ref{thm:recursive}, and show that the distance between $h\in \Heis$ and its approximants $\Frac \{\gamma_i\}_{i=1}^n$ satisfies a variant of a classical relation in Theorem \ref{thm:classicalformula}. We prove that the convergence of $\Frac CF(h)$ is uniform on a full-measure set in Theorem \ref{thm:uniformity}.

In \S \ref{sec:gauss} we consider a generalization of the classical Gauss map $x \mapsto 1/x - \floor{1/x}$. Namely, let $K_D \subset \Heis$ be the Dirichlet region for $\Heis(\Z)$, defined as the set of points $h$ such that $[h]=0$. It is easy to see that $K_D$ is a fundamental region for $\Heis(\Z)$, that is, the translates of $K_D$ by elements of $\Heis(\Z)$ tile $\Heis$ without overlap. It is also clear that for all $h\in \Heis$, one has $[h]^{-1}h \in K_D$. 

We define a function $T: K_D \rightarrow K_D$ on the Dirichlet region by $T(h) = [\iota h]^{-1} \iota h$. We ask whether $K$ admits a $T$-invariant measure absolutely continuous with respect to Lebesgue measure, and whether $T$ is ergodic with respect to this measure (Questions \ref{q:invariant} and \ref{q:ergodic}). We demonstrate the difficulties in studying these questions by discussing continued fractions on the complex plane $\C$, and answer them positively for the simpler base-$b$ expansions in $\Heis$.

We will now recall some background on classical continued fractions (\S \ref{sec:classicalcf}) and the Heisenberg group (\S \ref{sec:heis}), and then study Heisenberg continued fractions in \S \ref{sec:heiscf} and discuss their dynamical properties in \S \ref{sec:gauss}.

%%%%%%%%%%%%%%%%%%%%%%%%%%%%%%%%%%%%%%%%%%%%%%%%%%%%%%%%%
\subsection{Classical Theory of Continued Fractions}
\label{sec:classicalcf}
%%%%%%%%%%%%%%%%%%%%%%%%%%%%%%%%%%%%%%%%%%%%%%%%%%%%%%%%%

There are many variants on classical continued fractions and many ways to approach them (for good general references, see \cite{DHKM,Hensley,Khinchin,Masarotto}). We shall examine Nakada's $\alpha$-continued fractions, since the study of them bears the most immediate resemblance to the Heisenberg continued fractions we examine in this paper (see also \S \ref{sec:dynamicaltools} for a discussion of continued fractions on $\C$).  The $\alpha$-continued fractions have two well-known continued fraction variants as special cases: Regular Continued Fractions (when $\alpha=1$) and Nearest Integer Continued Fractions (when $\alpha=1/2$).

Let $\alpha\in(0,1]$. Define the the $\alpha$-Gauss map $T_\alpha: [\alpha-1,\alpha) \to [\alpha-1,\alpha)$ by
\(
T_\alpha x := \begin{cases}
x^{-1}-\left[x^{-1}\right]_\alpha, & x\neq 0,\\
0, & x=0,
\end{cases}
\)
where $[x]_{\alpha}$ is the unique integer such that $x-[x]_\alpha \in [\alpha-1,\alpha).$ Most continued fractions variants begin with these three simple pieces: a fundamental domain ($[\alpha-1,\alpha)$ here), an inversion that takes a point out of the fundamental domain ($x^{-1}$), and a piecewise linear translation that shifts us back into the fundamental domain ($-[x^{-1}]_\alpha$).

The digits of the $\alpha$-continued fraction expansion for a number $x\in[\alpha-1,\alpha)$ consist of two parts, $(a_n,\epsilon_n)$, where
\(
a_n= a_n(x)=\left[ T_{\alpha}^{n-1} x \right]_\alpha \quad \text{ and } \quad
\epsilon_n = \epsilon_n(x) = \operatorname{sgn}(T_{\alpha}^{n-1} x).
\)
The sequence of digits $(a_n,\epsilon_n)$ terminates when $T_\alpha^n x = 0$. These digits serve to record the data that is lost by iterating the non-injective map $T_\alpha$. In particular, we have
\(
x = \frac{\epsilon_1}{a_1 + T_\alpha x} = \cfrac{\epsilon_1}{a_1+\cfrac{\epsilon_2}{a_2+T_\alpha^2 x}} = \cdots.
\)
Note that $(a_n(x),\epsilon_n(x))=(a_{n-1}(T_\alpha x),\epsilon_{n-1}(T_\alpha x))$, so that $T_\alpha$ acts as a forward shift of the continued fraction digits of $x$.

One of the fundamental objects of study in the field of continued fractions is the sequence of \emph{convergents} or \emph{rational approximants} for a number $x$, given by
\(
\frac{p_n}{q_n} := \cfrac{\epsilon_1}{a_1+\cfrac{\epsilon_2}{a_2+\dots+\cfrac{\epsilon_n}{a_n}}}.
\)
It is often easier to understand abstract properties of the sequence of convergents for a number $x$, than it is to understand abstract properties of the whole continued fraction expansion for $x$.

A particularly useful property of convergents is the following matrix relation:
\[\label{eq:classicalmatrixrelation}
\left( \begin{array}{cc} p_{n-1} & p_n \\ q_{n-1} & q_n \end{array} \right) =
\left( \begin{array}{cc} 0 & \epsilon_1 \\ 1 & a_1 \end{array} \right)
\left( \begin{array}{cc} 0 & \epsilon_2 \\ 1 & a_2 \end{array} \right) \dots
\left( \begin{array}{cc} 0 & \epsilon_n \\ 1 & a_n \end{array} \right) .
\]
From this relation, one can derive the recurrence relation $q_n = a_n q_{n-1}+\epsilon_n q_{n-2}$.  While it would be nice to know that the $q_n$ form an increasing sequence of positive integers, this is not always the case (such as with continued fractions with odd partial quotients \cite{BocaVandehey}).

We can treat matrices as M\"{o}bius transforms, via
\(
\left( \begin{array}{cc} a & b \\ c & d \end{array} \right) z = \frac{az+b}{cz+d}.
\)
If we do this, then the simple relation
\(
x = \frac{\epsilon_1}{a_1+T_\alpha x} = \left( \begin{array}{cc} 0 & \epsilon_1\\ 1 & a_1 \end{array} \right) T_\alpha x,
\)
together with \eqref{eq:classicalmatrixrelation}, implies the more interesting relation
\[\label{eq:xinTx}
x = \left( \begin{array}{cc} p_{n-1} & p_n \\ q_{n-1} & q_n \end{array} \right) T_\alpha^{n} x = \frac{p_{n-1} T_{\alpha}^n x +p_n}{q_{n-1} T_\alpha^n x + q_n}.
\]
By solving for $T_\alpha^n x$ (or by applying the inverse of the matrix to both sides), one can obtain
\[\label{eq:Txinx}
T_\alpha^n x = (-1)\cdot \frac{q_n x-p_n}{q_{n-1} x-p_{n-1}}.
\]
Careful---but elementary---manipulation of the formulas \eqref{eq:xinTx} and \eqref{eq:Txinx} yields 
\[\label{eq:classicalrelation}
q_n x - p_n = (-1)^n \prod_{i=0}^n T_\alpha^n x = (-1)^n \cdot \frac{\epsilon_1 \epsilon_2 \cdots\epsilon_n}{ q_n T_\alpha^{n+1}x + q_{n+1}}.
\]

From \eqref{eq:classicalrelation} it is short exercise to see that $q_n x - p_n$ converges to $0$, and hence that $p_n/q_n$ converges to $x$.  Thus it makes sense to write $x$ as an infinite continued fraction expansion
\(
x = \cfrac{\epsilon_1}{a_1+\cfrac{\epsilon_2}{a_2+\dots}}.
\)

There are varying notions of convergence for continued fractions variants besides the fact that $|x-(p_n/q_n)|$ tends to $0$, which is typically known as \emph{weak convergence}.  In multi-dimensional continued fractions, where one might have convergents 
\(
\left( \frac{p_{1,n}}{q_n}, \frac{p_{2,n}}{q_n}, \dots, \frac{p_{k,n}}{q_n}\right) \text{ to a point } \left( x_1, x_2, \dots, x_k\right),
\)
the property that $|q_n x_i - p_{i,n}|$ tends to $0$ for all $i$ is known as \emph{strong convergence}.  (The Jacobi-Perron continued fraction, which is in many ways considered to be the prototypical multi-dimensional continued fraction, does not satisfy strong convergence.)  The fact that all columns of the matrices \eqref{eq:classicalmatrixrelation} converge (projectively) to the same point is known as \emph{uniform convergence}.  Uniform convergence is non-trivial for higher-dimensional continued fraction variants.

In general, it is hard to know whether an arbitrary sequence of continued fraction digits $(a_n,\epsilon_n)\in \mathbb{R}^2$ produces a convergent infinite continued fraction. (Even the seemingly innocuous two-digit sequence $\{(1,1),(1,-1)\}$ causes convergence problems.) One major result on this question is Pringsheim's theorem, which states that if $|a_n|\ge |\epsilon_n|+1$ for all $n \in\mathbb{N}$, then the infinite continued fraction converges.  For more on this topic, see \cite{Wall}.

For many continued fractions, the digit shift map $T$ is ergodic with respect to some invariant measure. For Regular Continued Fractions ($\alpha=1$), the invariant measure that is absolutely continuous with respect to Lebesgue is the classic Gauss measure
\(
\mu(A) = \frac{1}{\log 2} \int_A \frac{1}{1+x} \ dx.
\)
The ergodicity of the map $T$ means that there is a notion of average behavior for continued fractions. For example, for almost all $x$, the regular continued fraction expansion of $x$ satisfies
\(
\lim_{n\to \infty} \frac{\log q_n(x)}{n} = \frac{\pi^2}{12 \log 2}
\)
and $a_n=1$ approximately $42$ percent of the time.

Applications of continued fractions come from various areas. We mention only a few in greater detail here.  One of the most classical results on continued fractions is Lagrange's Theorem, which states that $x$ has an eventually periodic continued fraction expansion if and only if $x$ is a quadratic irrational number: thus one often studies properties of quadratic irrationals by understanding their RCF expansion.  The term $q_n x - p_n$ that appeared in \eqref{eq:classicalrelation} is closely related to the study of best approximants---namely, rational numbers $n/m$ that satsify the following relation
\(
|mx-n| \le |bx-a| \qquad a,b\in\mathbb{Z}, \quad1\le b < m
\)
must be an RCF convergent $p_n/q_n$ for $x$.

\section{The Heisenberg Group}
\label{sec:heis}
We will think of the Heisenberg group in three different ways. For geometric purposes, including illustration and discussion of measures, we will identify $\Heis$ with $\R^3$ (with the appropriate group structure and geometry). For the majority of the paper, however, we will be concerned with the representation of $\Heis$ as a group of unitary matrices or as a subset of $\C^2$. This is in direct analogy with thinking of the real numbers as elements of $SL(2,\R)$ or as the real axis within $\C^1$. We now discuss these models, and then record some information on discrete subgroups of $\Heis$ and their fundamental domains.

We emphasize that the topological and measure-theoretic notions we consider do not (qualitatively) depend on the model we choose, nor on the metric. In particular, convergence in $\Heis$ can be shown using the intrinsic gauge metric, or using metrics intrinsic to the model, such as the Euclidean metrics on $\R^3$ or $\C^2$.

\begin{figure}[h]
\label{fig:spheresattack}
\includegraphics[width=.45\textwidth]{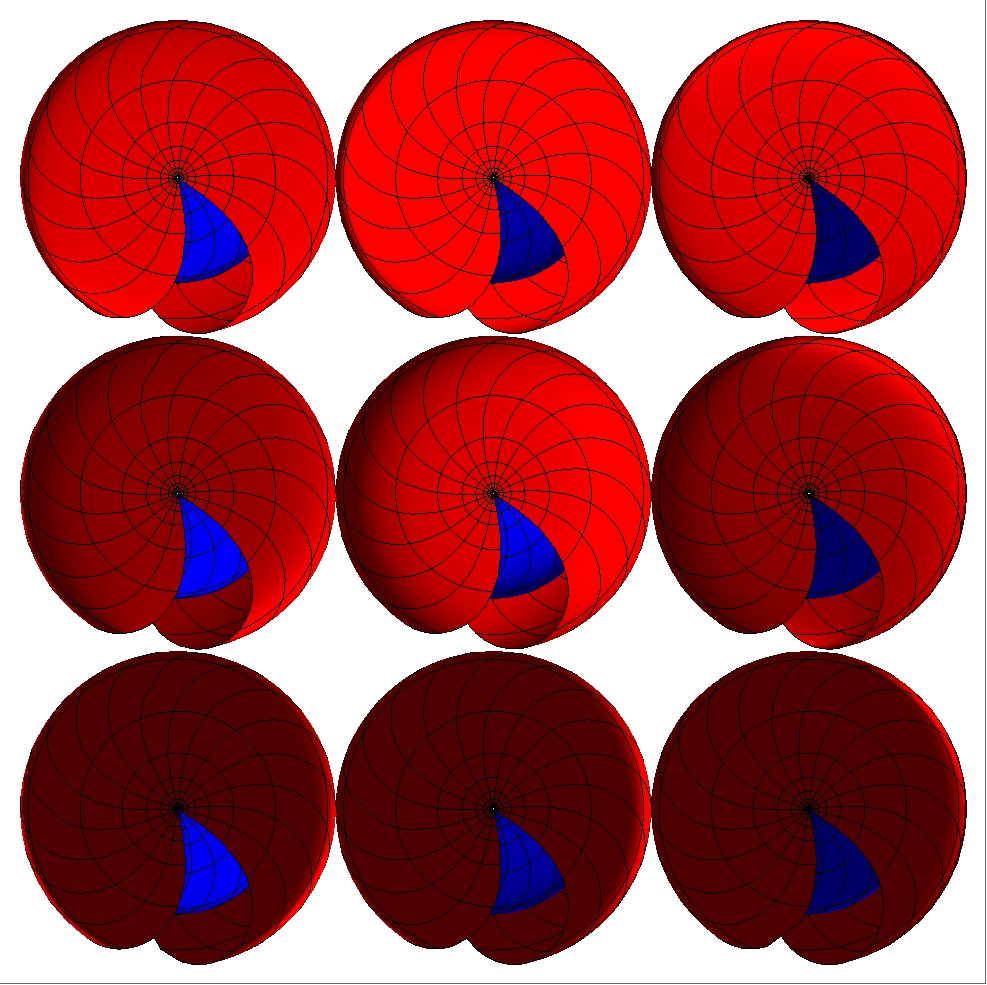}
\includegraphics[width=.45\textwidth]{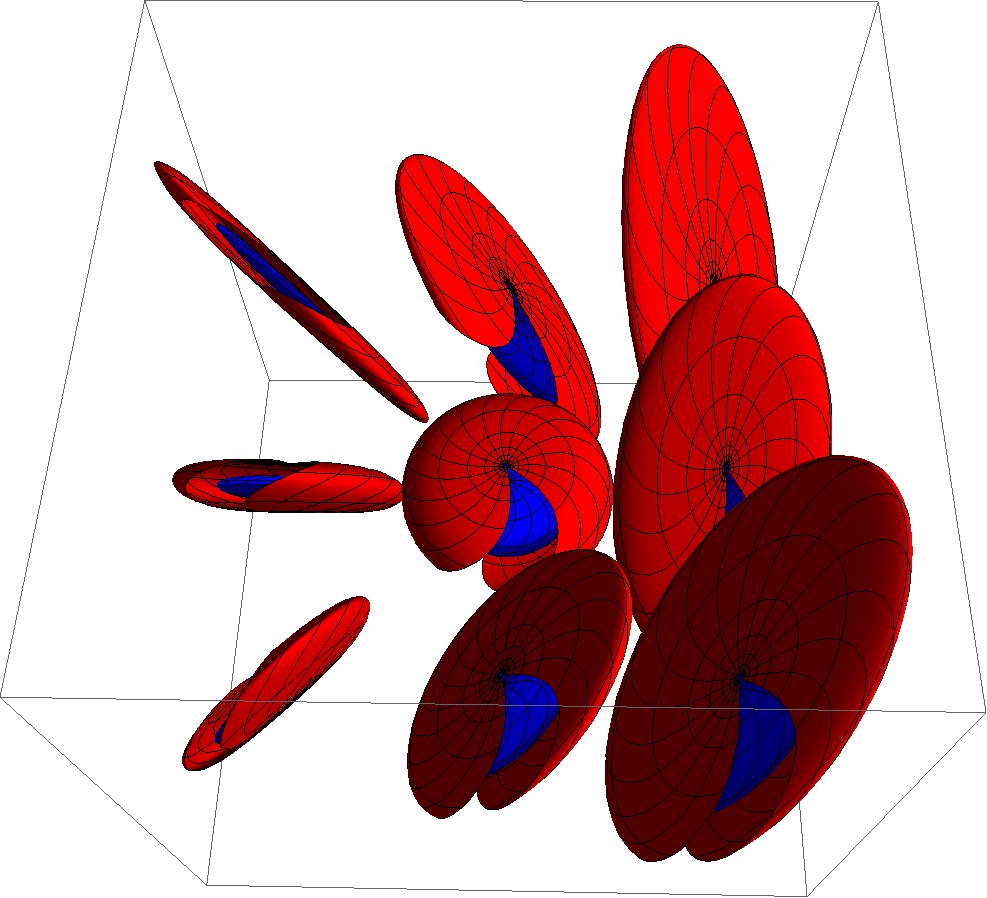}
\caption{Two views of nested spheres in $\Heis$, centered at $(i,j,0)$ with $i,j\in \{-1,0,1\}$, related to each other by left translation by elements of $\Heis(\Z)$. In the top view (left), the spheres look identical. A side view (right) shows an additional a shear in the $t$ coordinate.} 
\end{figure}
\subsection{Geometric Model}
\label{sec:geometric}
In the introduction, we defined $\Heis$ as the space $\R^3$ with group law
\((x,y,t)*(x',y',t') = (x+x',y+y', t+t'+2(xy'-yx')).\)

Combining the first two coordinates into a complex number, $\Heis$ becomes $\C\times \R$ with group law
\((z,t)*(z',t') = (z+z', t+t'+2\Im(\overline z z' )).\)

We will think of these as the same model, and use it primarily when geometry or visualization are concerned. There are several standard (topologically equivalent) metrics on $\Heis$; we will work with the gauge metric. The gauge $\Norm\cdot$ and distance $d$ are defined by:
\( &\Norm{(z,t)}=\sqrt[4]{\norm{z}^4+t^2}, &d(h,k)=\Norm{h^{-1}*k}, \quad h,k\in \Heis.\)

There are four basic transformations we are interested in:
\begin{enumerate}
\item Left translations $h\mapsto k*h$, for $k \in \Heis$,
\item Rotations $(z,t) \mapsto \left(e^{\ii \theta}z,t\right)$, for $\theta\in \R$,
\item Metric dilations $\delta_r: (z,t) \mapsto (rz, r^2t)$, for $r\in \R$,
\item The Koranyi inversion $\iota: \Heis\backslash\{0\} \rightarrow \Heis\backslash \{0\}$ given by
\(\iota(z,t) = \left( \frac{-z}{\norm{z}^2+\ii t}, \frac{-t}{\norm{z}^4+t^2}\right).\)
\end{enumerate}

Translations and rotations do not distort distances or volume (that is, the Lebesgue measure $\lambda$ on $\R^3$). The map $\delta_r$ is a group homomorphism dilating distances by a factor of $r$ and volume by a factor of $r^4$. The Koranyi inversion 
is a conformal map with the following important property.

\begin{lemma}[See p.19 of \cite{tysonetal}]
\label{lemma:iota}
Let $h,k\in \Heis\backslash\{0\}$. One has
\(d(\iota h, \iota k) = \frac{d(h,k)}{\Norm{h}\Norm{k}}.\)
\end{lemma}

In particular, one has $\Norm{\iota h} = \Norm{h}^{-1}$, so that the inside and outside of the unit ball are interchanged. Note that individual points on the unit sphere are not fixed.

\begin{remark}
We will show in Lemma \ref{lemma:koranyiinversion} that $\iota$ has a particularly simple form in the unitary model. It is conformal with respect to the gauge metric, see \cite{Koranyi-Reimann1995}.
\end{remark}

We record the following relationship between volumes and radii of balls in $\Heis$. In particular, the lemma implies that the Heisenberg group has Hausdorff dimension 4, and that the Lebesgue measure $\lambda$ is equivalent to the Hausdorff $4$-measure on $\Heis$.
\begin{lemma}
\label{lemma:ballvolume}
The volume of a ball $B(h,r)$ of radius $r$ around a point $h$ is given by
\(\lambda(B(h,r)) = r^4 \lambda(B(0,1))\)
\begin{proof}
Applying a left translation, we may assume $h$ is the origin. Further, we may rescale $B(0,r)$ by the Heisenberg dilation $\delta_{1/r}(z,t) = (z/r, t/{r^2})$ to obtain $B(0,1)$. The dilation distorts $\lambda$ by a factor of $r^{-4}$.
\end{proof}
\end{lemma}

An immediate consequence of Lemmas \ref{lemma:iota} and \ref{lemma:ballvolume} is the following form for the Jacobian of $\iota$:
\begin{lemma}
\label{lemma:Jacobianiota}
The Jacobian determinant of the Koranyi inversion $\iota$ at a point $h\in \Heis$ is given by
$J_h\iota = \Norm{h}^4$.
\end{lemma}

\begin{figure}[h]
\label{fig:nestedSpheres}
\includegraphics[width=.3\textwidth]{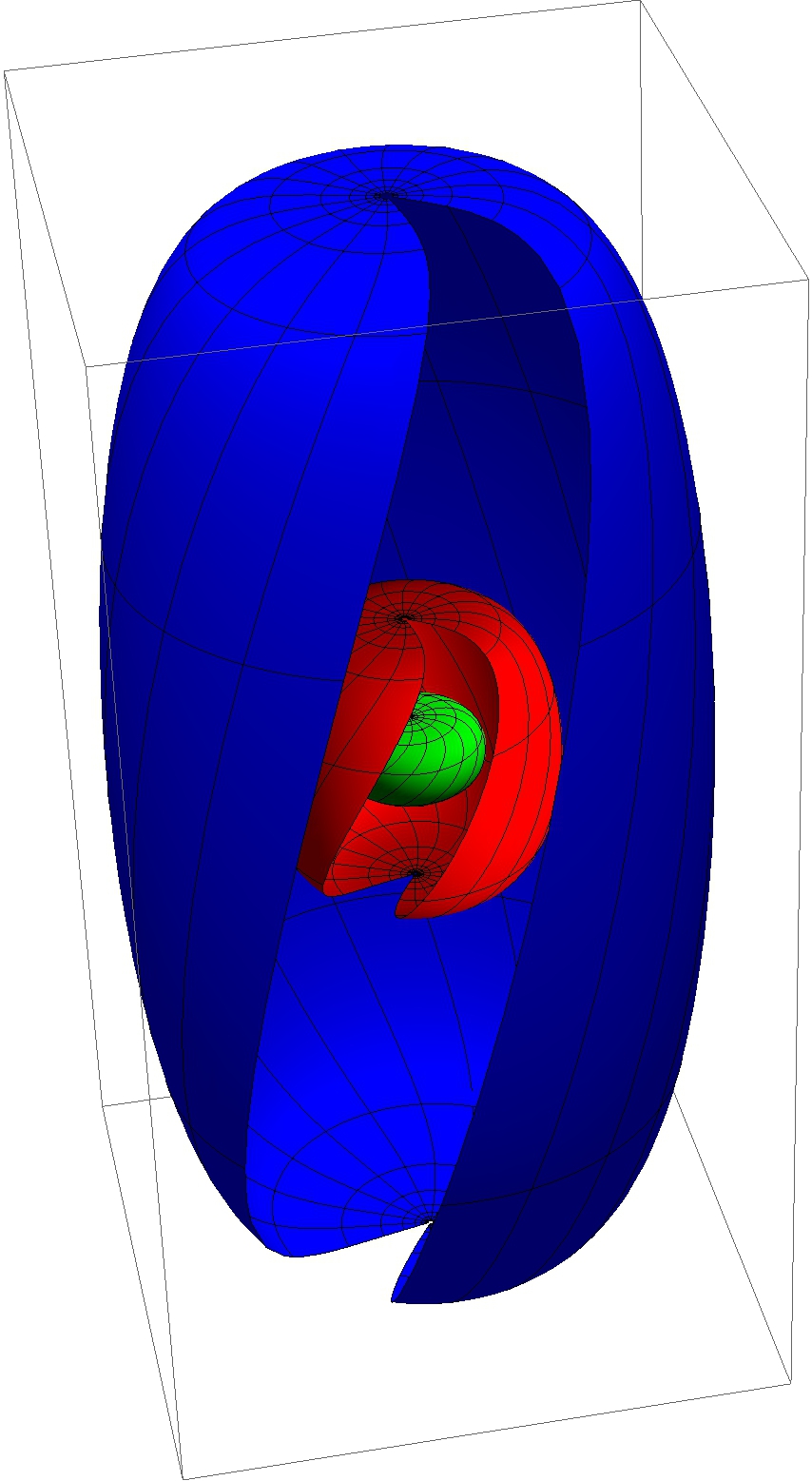}
\caption{Spheres in $\Heis$ centered at the origin, with radius $2, 1, 1/2$, with sectors removed to display nested spheres. The spheres are parametrized by applying $\iota$ to a plane; the radial lines of the plane provide the characteristic foliation on the spheres.}
\end{figure}

\subsection{Real Nilpotent Model}
It is common to describe $\Heis$ as the group of nilpotent upper-triangular 3-by-3 real matrices. Our definition is related to this \emph{real nilpotent model} via the Lie group isomorphism
\[\label{fla:real-nilpotent} (x,y,t)\mapsto 
\left(\begin{array}{ccc}
1 & x &\frac{t}{4}+\frac{x y}{2}\\
0 &1 & y\\
0 &0 &1
\end{array}\right).
\]
We will not use the real nilpotent model, although our results can be rephrased for it. Note that under (\ref{fla:real-nilpotent}), $\Heis(\Z)$ is \emph{not} identified with matrices with integer entries.

\subsection{Unitary Representation}
\label{sec:unitary}
For calculation purposes, we will use the  \emph{(Siegel) unitary representation} of $\Heis$. Namely, we will embed $\Heis$ in $GL(3,\C)$ via the homomorphism:
\[
\label{fla:HeisInclusion}
\mathbb U: (z,t) \mapsto  \left( \begin{array}{ccc}
1 & 0 & 0 \\
z (1+\ii)& 1 & 0 \\
\norm{z}^2+t \ii & \overline z (1-\ii)& 1 \end{array} \right).\]

\begin{remark}In literature, one sees a factor of $\sqrt{2}$ rather than $1+\ii$ in the embedding. The latter is more convenient for our purposes.\end{remark}

Let $\J$ be the Hermitian inner product given by
\( \J( (z_0, z_1, z_2), (w_0, w_1, w_2)) =  \left( \begin{array}{ccc}
\overline{z_0} &\overline{z_1} &\overline{z_2}\end{array} \right)
\left( \begin{array}{ccc}
0& 0 & -1 \\
0& 1 & 0 \\
-1&0& 0 \end{array} \right)\left( \begin{array}{ccc}
{w_0}\\
{w_1}\\
{w_2}\end{array} \right)
\)
In particular, we record
\[\label{eqn:norm} \norm{(z_0, z_1, z_2)}^2_\J=\J( (z_0, z_1, z_2), (z_0, z_1, z_2)) =\norm{z_1}^2-  2 \Re(\overline{z_0}z_2).\]
We will refer to a vector of norm $0$ as a \emph{null vector}.

Abusing notation, we will also use $\J$ to denote the skew-diagonal matrix above. Note that $\J$ has signature $(2,1)$: it has two positive and one negative eigenvalue.

The unitary group $U(2,1)\subset GL(3,\C)$ is the set  of matrices $M\in GL(3,\C)$ satisfying $\J(M\vec z, M\vec w) = \J(\vec z, \vec w)$ for all $\vec z, \vec w \in \C^3$. Equivalently, $M$ satisfies $M^\dagger \J M = \J$, where $\dagger$ denotes the conjugate transpose. We will additionally distinguish the subgroups $SU(2,1)$ and $S^{\pm}U(2,1)$ consisting of matrices $M\in U(2,1)$ satisfying, respectively, $\det M = 1$ or $\det M = \pm 1$. We have:
\begin{lemma}
$\mathbb U(\Heis) \subset SU(2,1)$.
\end{lemma}

Later calculations will require us to step outside of $\mathbb U(\Heis)$. The following lemma provides a basic property of elements of unitary matrices.

\begin{lemma}
\label{lemma:crossproduct}
Every matrix $M$ in $U(2,1)$ is of the form
\[
\left(\begin{array}{ccc}
a_{1,1} & \operatorname{det}(M)\cdot \overline{a_{2,3}a_{1,1}-a_{2,1}a_{1,3}} &a_{1,3}\\
a_{2,1} &  \operatorname{det}(M)\cdot\overline{a_{3,3}a_{1,1}-a_{3,1}a_{1,3}} &a_{2,3}\\
a_{3,1} &  \operatorname{det}(M)\cdot\overline{a_{3,3}a_{2,1}-a_{3,1}a_{2,3}} &a_{3,3}
\end{array}\right).
\]
\begin{proof}
Every matrix $M=(a_{i,j})$ in $U(2,1)$ satisfies $M^\dagger \J = \J M^{-1}$. We also have
\( 
M^\dagger \J= \left( \begin{array}{ccc}
-\overline{a_{3,1}} & \overline{a_{2,1}}  & -\overline{a_{1,1}} \\
-\overline{a_{3,2}} & \overline{a_{2,2}} & -\overline{a_{1,2}} \\
-\overline{a_{3,3}} & \overline{a_{2,3}} & -\overline{a_{1,3}} \end{array} \right).
\)
On the other hand, 
\(\J M^{-1} = 
 \operatorname{det}(M)\left( \begin{array}{ccc}
a_{3,1}a_{2,2}-a_{3,2}a_{2,1} & a_{3,2}a_{1,1}-a_{3,1}a_{1,2}  & a_{2,1}a_{1,2}-a_{2,2}a_{1,1} \\
a_{3,1}a_{2,3}-a_{3,3}a_{2,1} & a_{3,3}a_{1,1}-a_{3,1}a_{1,3}  & a_{2,1}a_{1,3}-a_{2,3}a_{1,1} \\
a_{3,2}a_{2,3}-a_{3,3}a_{2,2} & a_{3,3}a_{1,2}-a_{3,2}a_{1,3}  & a_{2,2}a_{1,3}-a_{2,3}a_{1,2}  \end{array} \right).
\)
Comparing the two matrices completes the lemma.
\end{proof}
\end{lemma}

\subsection{Siegel Model}
\label{sec:siegel}
The Siegel model provides a geometric view of the unitary representation and a simpler formula for the Koranyi inversion. We will in fact define two closely related models, the \emph{planar Siegel model} that views a point $h\in \Heis$ as a vector $(u,v)\in \C^2$, and the \emph{projective Siegel model} that views $h$ as a point in complex projective space with homogeneous coordinates $(1:u:v)$. We will denote both models by $\Sieg$.

We first identify a point $h\in \Heis$ with geometric coordinates $(z,t)$ with the vector 
\[\label{fla:siegelForm}\left(1, z(1+\ii), \norm{z}^2+\ii t\right)\in \C^3.\]
 Note that this is exactly the image of the vector $(1,0,0)$ under the unitary transformation $\mathbb U(z,t)$. We will say that $h$ has \emph{planar Siegel coordinates} 
\[
\label{fla:planarSiegel}
(z(1+\ii), \norm{z}^2+\ii t) \in \C^2.\]
 The \emph{planar Siegel model} of $\Heis$ is the set of points in $\C^2$ of the form (\ref{fla:planarSiegel}).

Sometimes, a unitary transformation will take $(1, z(1+\ii), \norm{z}^2+\ii t)$ to a point that is not of the same form, but can be rescaled to be such. It will therefore be useful to think of vectors up to rescaling, that is, as elements of complex projective space $\CP^2$.

Recall that the complex projective plane $\CP^2$ is the projectivization of $\C^3$, i.e.\ the set of non-zero vectors up to rescaling by a non-zero complex number. A point in $\CP^2$ has \emph{homogeneous coordinates} $(z_0 : z_1: z_2)$, well-defined up to rescaling.  

We can now define the \emph{projective Siegel model} of $\Heis$ as the set of points in $\CP^2$ with homogeneous coordinates $(1: z(1+\ii): \norm{z}^2+\ii t)$.

Abusing notation, we will denote both Siegel models by $\Sieg$, with the identification $(u,v)\leftrightarrow (1:u:v)$. We have the following simple characterization of points in $\Sieg$.

\begin{lemma}
Let $(z_0:z_1:z_2) \in \CP^2$ be a null point, that is $\Norm{(z_0, z_1, z_2)}^2_{\mathbb J}=0$. Then either $(z_0:z_1:z_2) \in \Sieg$ or $(z_0:z_1:z_2)\cong(0:0:1)$.
\end{lemma}

We denote the closure of $\Sieg$ in $\CP^2$ by $\overline\Sieg=\Sieg\cup \{(0:0:1)\}$.

\begin{remark}
\label{rmk:complexhyperbolic}
The region $\{(z_0:z_1:z_2)\in \CP^2 \st \Norm{(z_0, z_1, z_2)}^2_{\J}<0\}$ bounded by $\overline \Sieg$ is the \emph{Siegel domain}. Complex hyperbolic space is defined on this region and has strong connections to the Heisenberg group, see e.g.\ \cite{tysonetal,goldman,  Koranyi-Reimann1995, lukyanenko2012}. In particular, we intend to discuss the relation of Heisenberg continued fractions to geodesic coding in complex hyperbolic space in an upcoming paper, following \cite{MR810563}.
\end{remark}

Note that the gauge norm is easy to write in the Siegel model:
\begin{lemma}
\label{lemma:gaugeSiegel}
Let $(u,v) \in \Sieg$. Then the gauge norm of $(u,v)$ is $\Norm{(u,v)} = \norm{v}^{1/2}$.
\begin{proof}
An element of $\Sieg$ has the form $(u,v)=\left(z(1+\ii), \norm{z}^2+t\ii\right)$ for some $(z,t)\in\Heis$. The gauge norm of $(z,t)$ is given by $\Norm{(z,t)}=\sqrt[4]{\norm{z}^4+t^2}=\norm{v}^{1/2}$.
\end{proof}
\end{lemma}

The gauge distance is defined as $d(h,k)=\Norm{h^{-1}k}$. With this in mind, we show:
\begin{lemma}
\label{lemma:siegelMinus}
In the planar Siegel model, we have
\( (u_1, v_1)^{-1}(u_2, v_2) = (u_2-u_1, \overline{v_1}-\overline{u_1}u_2+v_2).\)
\begin{proof}
We have associated to $(u_1, v_1)^{-1}$ the matrix
\(
\left(
\begin{array}{ccc}
1&0&0\\
-u_1 &1&0\\
\overline{v_1} &-\overline{u_1}&1
\end{array}
\right)
\)
Applying this matrix to the point $(1, u_2, v_2)$, we get the vector \((1, u_2-u_1, \overline{v_1}-\overline{u_1}u_2+v_2).\) Taking the last two coordinates yields the
desired formula.
\end{proof}
\end{lemma}

We now study the action of $S^\pm U(2,1)$ matrices on the Heisenberg group in the Siegel models. General linear matrices act on $\CP^2$ by acting on the homogeneous coordinates. Since we have $\C^2 \hookrightarrow \CP^2$ by taking $(u,v)\mapsto (1:u:v)$, we also obtain an action on $\C^2$.

\begin{lemma}
Let $M=(a_{i,j}) \in GL(3,\C)$ and $(u,v)\in \C^2 \hookrightarrow \CP^2$. Then $M$ acts on $(u,v)$ as:
\(M(u,v) = \left( \frac{a_{2,1} + a_{2,2}u+a_{2,3}v}{a_{1,1}+a_{1,2}u+a_{1,3}v}, \frac{a_{3,1}+a_{3,2}u+a_{3,3}v}{a_{1,1}+a_{1,2}u+a_{1,3}v}\right).\)
\begin{proof}
The point $(u,v)$ corresponds to a point in $\CP^2$ with homogeneous coordinates $(1:u:v)$. We then have 
\(
M\left(\begin{array}{ccc}1\\u\\v\end{array}\right)=\left(\begin{array}{ccc}a_{1,1}+a_{1,2}u+a_{1,3}v\\a_{2,1} + a_{2,2}u+a_{2,3}v\\a_{3,1}+a_{3,2}u+a_{3,3}v\end{array}\right)
\)
To view $M(1:u:v)$ as a point in $\C^2$, we renormalize so that the first coordinate is 1, and take the remaining two coordinates.
\end{proof}
\end{lemma}

 Elements of $GL(3,\C)$ do not necessarily preserve the set $\overline \Sieg$, but the unitary matrices $U(2,1)$ preserve $\J$ and therefore $\overline \Sieg$. In particular, elements of $\mathbb U(\Heis)$ act transitively on $\Sieg$ while fixing the point $(0:0:1)$. Denote the matrix $\left(\begin{array}{ccc}0&0&-1\\0&1&0\\-1&0&0\end{array}\right)$ by $\mathbb U(\iota)$.

\begin{lemma}
\label{lemma:koranyiinversion}
$\mathbb U(\iota)$ acts on $\Heis$ by the Koranyi inversion $\iota$.
\begin{proof}
We compute, for a point in $\Heis$ with geometric coordinates $(z,t)$ and projective Siegel coordinates $(1:z(1+\ii): \norm{z}^2+t\ii)$:
\(\mathbb U(\iota)(1:z(1+\ii):\norm{z}^2+t\ii) &= (\norm{z}^2+t\ii: -z(1+\ii):1) \\
&= \left(1: \frac{-z}{\norm{z}^2+t\ii}(1+\ii): \frac{1}{\norm{z}^2+t\ii}\right)\\
&= \left(1: \frac{-z}{\norm{z}^2+t\ii}(1+\ii): \frac{\norm{z}^2-t\ii}{\norm{z}^4+t^2}\right)\\
&= \left(1: \frac{-z}{\norm{z}^2+t\ii}(1+\ii): \norm{\frac{-z}{\norm{z}^2+t\ii}}^2+\frac{-t}{\norm{z}^4+t^2}\ii\right)
\) 
We thus have that under $\mathbb U(\iota)$, the geometric coordinates $(z,t)$ are mapped to $\left(\frac{-z}{\norm{z}^2+t\ii}, \frac{-t}{\Norm{(z,t)}^4}\right)$, as desired.
\end{proof}
\end{lemma}

\subsection{Lattices and Fundamental Domains}\label{section:lattices}

Recall that $\Heis(\Z)$ is the set of Heisenberg points with integer coordinates. In the geometric model $\Heis=\C\times\R$, we have $\Heis(\Z)=\Z[\ii]\times\Z$. In the Siegel model, $\Heis(\Z)$ is the set of points $(u,v)\in \Sieg$ such that $u\in (1+\ii)\Z[\ii], v\in \Z[\ii]$. In the unitary model, we have $\Heis(\Z) \subset SU(2,1;\Z[\ii])$, where the latter denotes the subset of $SU(2,1)$ with Gaussian integer coefficients, and is known as the Picard modular group.

Likewise, we will denote by $\Heis(\Q)$ the set of points in $\Heis$ with rational coordinates. Recall that the Heisenberg group admits a family of dilation maps $\delta_r$ given by $\delta_r(z,t) = (rz, r^2t)$ in the geometric model. The dilation maps are group isomorphisms and satisfy $d(\delta_r h, \delta_r q) = r\cdot d(h,k)$ for all $h,k\in \Heis$ and $r\geq0$. It is clear that $h\in \Heis(\Q)$ if and only if there is an integer $n\in \N$ such that $\delta_n h \in \Heis(\Z)$.

We are now interested in the structure and geometry of $\Heis(\Z)$. We record its generators in the geometric model:
\begin{lemma}
The group $\Heis(\Z)$ is generated by the elements $(1,0)$, $(\ii, 0)$, and $(0,1)$.
\end{lemma}

As Falbel--Francics--Lax--Parker showed, $\Heis[\Z]$ and $SU(2,1;\Z[\ii])$ are closely linked (see also \cite{MR2900617}):

\begin{thm}[\cite{falbel-et-al-2011-proceedings}]
The group $SU(2,1;\Z[\ii])$ is generated by the matrices $\mathbb U(1,0)$, $\mathbb U(0,1)$, $-\mathbb U(\iota)$, and the matrix
\(
\left(\begin{array}{ccc}
\ii &0&0\\
0&-1&0\\
0&0&\ii\end{array}\right)
\)
corresponding to the mapping $(z,t)\mapsto (-\ii z,t)$.
\end{thm}

We now discuss fundamental domains for $\Heis(\Z)$. Recall that a fundamental domain for $\Heis(\Z)$ is a connected set $K\subset \Heis$ with piecewise smooth boundary whose translates tile $\Heis$ essentially without overlap. That is, $\cup \{\gamma*K\st \gamma\in \Heis(\Z)\}=\Heis$ and $\interior{K}\cap (\gamma*\interior K)\neq \emptyset$ implies $\gamma=0$.

We require a slightly different definition. We require $K$ to consist of an open set and some measurable subset of its boundary (which is not necessarily piecewise smooth) such that $\cup \{\gamma*K\st \gamma\in \Heis(\Z)\}=\Heis$ and $K\cap (\gamma*K)\neq \emptyset$ implies $\gamma=0$. We then have:

\begin{lemma}
Let $K$ be a fundamental domain for $\Heis(\Z)$. Then the map $\class{p}_K: \Heis \rightarrow \Heis(\Z)$ mapping all points of $\gamma K$ to $\gamma$ is well-defined.
\end{lemma}

The following lemma follows immediately from the definitions:

\begin{lemma}
\label{lemma:fd}
The following regions are fundamental domains for $\Heis(\Z)$:
\begin{itemize}
\item The unit cube $K_C=[-1/2,1/2)\times[-1/2,1/2)\times[-1/2,1/2)$. 
\item The Dirichlet domain $K_D = \{h \in \Heis \st  d(0,h)\leq d(g,h) \text{ for all }g\in \Heis(\Z)\}$, with a choice of excluded boundary points.
\end{itemize}
\end{lemma}

\begin{figure}[h]
\includegraphics[width=.5\textwidth]{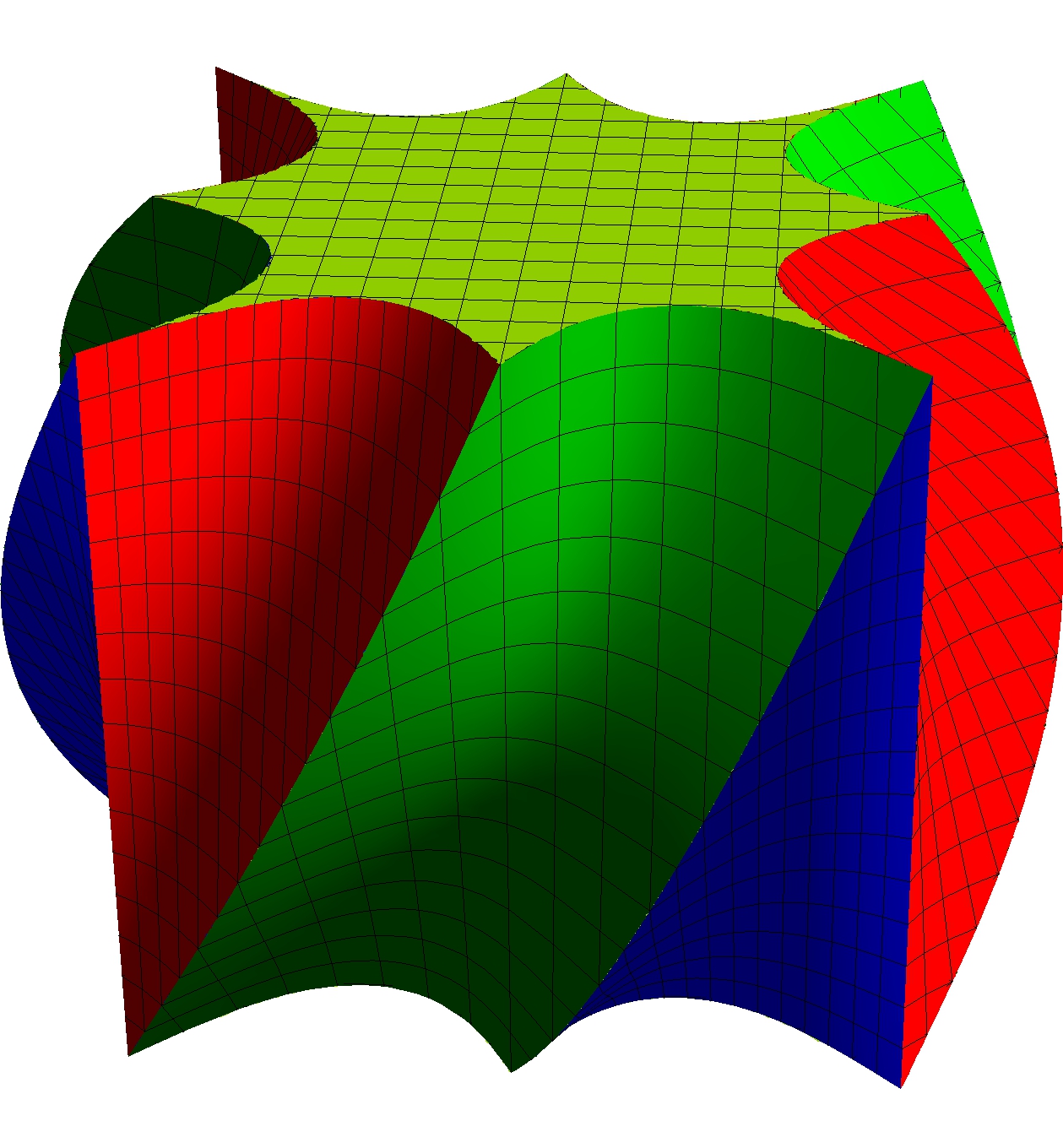}
\caption{The Dirichlet domain for $\Heis(\Z)$ centered at the origin.}
\label{fig:dirichlet}
\end{figure}

Denote the unit sphere in $\Heis$ by $S$. For a subset $A \subset \Heis$, let $\rad(A)$ denote the supremum of the norms of the points of $A$, and let $\lambda(A)$ denote its Lebesgue measure (in the geometric model).

\begin{lemma}
\label{lemma:radii}
Every fundamental domain $K$ for $\Heis(\Z)$ satisfies $\lambda(K)=1$. Furthermore, the domains $K_C$ and $K_D$ satisfy $\rad(K_C)=\rad(K_D)=\sqrt[4]{1/2}$.
\begin{proof}
The radius of $K_C$ is easy to compute because $\Norm{\cdot}$ behaves similarly to the Euclidean norm. As in the Euclidean case, the norm is maximized by each corner of the cube. We have $\Norm{(1/2+\ii1/2,1/2)}=\sqrt[4]{1/2}$.

The radius of $K_D$ seems difficult to compute directly, as the boundary of $K_D$ is more complicated (see Figure \ref{fig:dirichlet}). We will therefore argue indirectly by means of $K_C$. Let $h\in K_D$, and choose $g\in \Heis(\Z)$ so that $g*h\in K_C$. We then have $\Norm{g*h}\leq \rad(K_C)= \sqrt[4]{1/2}$. This implies that $d(g^{-1},h)\leq \sqrt[4]{1/2}$. Now, by definition of $K_D$, $d(0,h)\leq d(g^{-1},h) \leq \sqrt[4]{1/2}$, so we must also have $\Norm{h}\leq \sqrt[4]{1/2}$, so $\rad(K_D)\leq \sqrt[4]{1/2}$.
To prove equality, one shows that the point $(1/2+\ii 1/2, 1/2)$ is on the boundary of $K_D$.

For the volume computation, it is clear that $\lambda(K_C)=1$. To compute $\lambda(K)$ for an arbitrary fundamental domain $K$, note that Lebesgue measure is preserved by left translation in the Heisenberg group (which acts by shears). Since $K_C$ can be constructed by rearranging measurable pieces of $K$, the two fundamental domains must have the same volume.
\end{proof}
\end{lemma}

\begin{remark}
\label{rmk:su21z}
Note that we defined $U(2,1;\Z[\ii])$ with a particular Hermitian form $\J$ in mind. Different Hermitian forms $\J$ provide isomorphic Lie groups $U(2,1)$, but the lattice $U(2,1; \Z[\ii])$ depends on the choice of the Hermitian form. If two forms are related by an integer change of coordinates, then the associated lattices are equivalent. If the change of coordinates is not integral, the lattices are not isomorphic as groups (even up to finite index), see \cite{MR2527560, MR770063} Nonetheless, in literature one mostly sees mention of the Picard modular group, defined by a Hermitian form equivalent to our $\J$.
\end{remark}

\section{Heisenberg Continued Fractions}
\label{sec:heiscf}

Fix a fundamental domain $K$ for the group $\Heis(\Z)$ such that $\rad(K)<1$ (e.g.\ $K_C$ or $K_D$ in Lemma \ref{lemma:fd}). We begin by establishing some notation.

\begin{defi}
Given an arbitrary sequence $\{\gamma_i\}_{i=1}^n$ of non-zero digits in $\Heis(\Z)$, we write the associated continued fraction as
\[\Frac \{\gamma_i\} = \Frac\{\gamma_i\}_{i=1}^n= \iota \gamma_1 \iota \gamma_2 \cdots \iota \gamma_n.\]
For an infinite sequence $\{\gamma_i\}_{i=1}^\infty$, we define $\Frac \{\gamma_i\} = \Frac\{\gamma_i\}_{i=1}^\infty:= \lim_{n \rightarrow \infty} \Frac \{\gamma_i\}_{i=1}^n$, if this limit exists. The goal of this section is to show that the limit does exist in several important cases, and that the computation of $\Frac \{\gamma_i\}$ may be simplified by using a recursive algorithm.
\end{defi}

\begin{defi}
We associate with $K$:
\begin{enumerate}
\item A ``nearest-integer'' map $[\cdot]:\Heis \rightarrow \Heis(\Z)$, characterized by
\( [h]=g \text{ for each }g\in \Heis(\Z)\text{ and }
h \in gK.\)
Note that $[\cdot]$ selects the nearest integer in the gauge metric exactly if $K$ is the Dirichlet domain $K_D$.
\item The \emph{Gauss map} $T: K\backslash\{0\} \rightarrow K$ given by 
\(T(h) = [\iota h]^{-1}\iota h.\)
\end{enumerate}
\end{defi}

\begin{remark}
Working with the geometric model, one sees that each axis is preserved by the Gauss map $T$. In particular, the restriction of $T$ to each axis is essentially isomorphic to the nearest-integer Gauss map on $[-1/2,1/2]$. The theory of continued fractions we develop likewise restricts to the classical nearest-integer continued fraction theory on the axes.
\end{remark}

\begin{defi}
Given a point $h\in K$, have:
\begin{enumerate}
\item The \emph{forward iterates} $h_i := T^i h \in K$, for each $i$,
\item The \emph{continued fraction digits} $\gamma_i := [\iota h_{i-1}] \in \Heis(\Z)$, for each $i$,
\item The \emph{rational approximants} $\Frac \{\gamma_i\}_{i=1}^n \in \Heis(\Q)$, for each $n$.
\end{enumerate}
Because $T$ is defined on $K \backslash \{0\}$, the process of defining forward iterates, continued fraction digits, and rational approximants terminates if for some $i$ we have $h_i =0$. We will characterize the points $h$ for which this happens in Theorem \ref{thm:finiterational}.

More generally, for a point $h\in\Heis$ we can take $\gamma_0=[h]$, $h_0=\gamma_0^{-1}h$ and obtain the remaining digits $\{\gamma_i\}_{i=1}^\infty$ of $CF(h)$ from $h_0\in K$ as before.  However, our focus will be on points in $K$.
\end{defi}

It is easy to see that, on finite sequences, $\Frac$ is the inverse operation to $CF$:
\begin{lemma}
\label{lemma:KCF}
For $h\in K$ with $CF(h)$ a finite sequence, we have $\Frac CF(h)=h$.
\end{lemma}

\begin{remark}\label{remark:twoexpansions}
The operation $\Frac$ is defined without reference to a specific fundamental domain $K$. Thus, while we will show that $\Frac CF(h)=h$, we do not in general have $CF(\Frac \{\gamma_i\}) = \{\gamma_i\}$.  Indeed, problems arise when the $\gamma_i$ get too close to the unit sphere.

For example, let $K=K_C$, the unit cube, and $\{\gamma_i\} = \{(a_1,b_1)=(1,0)\}$. We have 
\(\Frac \{\gamma_i\} = \iota (1,0) = (-1,0).\) Attempting to reverse the process, we have $(a_0, b_0)=[(-1,0)] = (-1,0)$, and $(-1,0)^{-1}*(-1,0)=(0,0)$, so that 
\(CF(-1,0) = \{(a_0,b_0)=(-1,0)\}.\)

This non-uniqueness of continued fraction expansions is analogous to how in regular continued fractions we have, for example,
\(\frac{1}{5+\frac{1}{1}}=\frac{1}{6}.\)

\end{remark}

\subsection{Pringsheim-Type Theorem}
The Pringsheim Theorem for regular continued fractions guarantees convergence of a continued fraction whose digits are sufficiently large. A variant holds for the Heisenberg group:
\begin{thm}[Pringsheim-Type Theorem]
\label{thm:Pringsheim}
Let $\{\gamma_i\}_{i=1}^\infty$ be a sequence of points in $\Heis(\Z)$ such that for each $i$ we have $\Norm{\gamma_i}\geq 3$. Then the limit $\Frac \{\gamma_i\}$ exists. Furthermore, $CF( \Frac \{\gamma_i\}) = \{\gamma_i\}$.
\begin{proof}

Recall that left multiplication by any $\gamma\in \Heis(\Z)$ is an isometry, and that $\iota$ satisfies the relation $d(\iota h,\iota k)=\frac{d(h,k)}{\Norm{h}\Norm{k}}$  for all $h,k\in\Heis$ (Lemma \ref{lemma:koranyiinversion}).

Let $K_D$ be the Dirichlet fundamental domain for $\Heis(\Z)$. It follows from the definition of $K_D$ and the triangle inequality that for each point $h\in \Heis$ with $\Norm h<1/2$, we have $h\in K_D$. Conversely, for each point $h\in K_D$ we have by Lemma \ref{lemma:radii} that $\Norm{h}\leq \sqrt[4]{1/2}$ .

Suppose that $\gamma\in \Heis(\Z)$ with $\Norm{\gamma}\geq 3$. We claim that $\iota \gamma K_D \subset K_D$. Indeed, every point $h\in \gamma K_D$ satisfies $\Norm{\gamma h}\geq 3-\Norm{h}\geq  3-\sqrt[4]{1/2}>2$, so that $\Norm{\iota\gamma h} < \frac{1}{2}$, and we conclude $\iota \gamma K_D \subset K_D$.

Now, for each $n$, we have (because the identity element $0$ is contained in $K_D$):
\(\Frac \{\gamma_i\}_{i=1}^n  &= \iota \gamma_1 \iota \gamma_2  \cdots \iota \gamma_n \\ &= \iota \gamma_1 \iota \gamma_2  \cdots \iota \gamma_n 0\\
&\in \iota \gamma_1 \iota \gamma_2  \cdots \iota \gamma_n K_D.\)

These \emph{cylinder sets} form a nested sequence: 
\(
\iota \gamma_1 \iota \gamma_2  \cdots \iota \gamma_n K_D &= \iota \gamma_1 \iota \gamma_2  \cdots \iota \gamma_{n-1} (\iota \gamma_n K_D) \\&\subset \iota \gamma_1 \iota \gamma_2  \cdots \iota \gamma_{n-1}K_D.
\)

By the above calculation, the diameter of the cylinder set $\iota \gamma_1 \iota \gamma_2 \cdots \iota \gamma_{n} K_D$ is bounded above by $(3-\sqrt[4]{1/2})^{-2n} \text{diam}(K_D)$. We thus have that the sequence of fractions $\mathbb K \{\gamma_i\}_{i=1}^n$ (as $n$ varies) is a Cauchy sequence, and hence converges to some $\Frac \{\gamma_i\}$.

We thus have that $\Frac \{\gamma_i\}$ exists. By construction, we also know that it is contained in the cylinder sets $\iota \gamma_1 \iota \gamma_2  \cdots \iota \gamma_{n}K_D$, for each $n$ (note that the cylinder sets are in fact properly nested, so that $\Frac \{\gamma_i\}$ cannot escape to a cylinder set's boundary). This is equivalent to the second assertion of the theorem.
\end{proof}

\end{thm}

\subsection{Rational Points}
We will now show that a point in $\Heis$ has rational coordinates if and only if it has a finite continued fraction expansion. Our proof is motivated by the work of Falbel--Francsics--Lax--Parker \cite{falbel-et-al-2011-proceedings} and uses the Siegel model. 

Recall that for a point $h\in K$ that is of interest to us, we will write \(h=(u,v)\in \C^2\) in the planar Siegel model. We will also think of $(u,v)$ as the element of $\CP^2$ with homogeneous coordinates $(1:u:v)$. That is, it is the vector $(1,u,v)$ considered up to multiplication by a non-zero compex number.

\begin{defi}
\label{defi:Agamma}
Given an element $\gamma \in \Heis(\Z)$ with planar Siegel coordinates $(\alpha, \beta)\in(\Z[\ii]\times \Z[\ii])\cap \mathcal S$, define
\(
A_\gamma := \mathbb U(\iota) \mathbb U(\gamma)&=
\left(\begin{array}{ccc}
0 & 0 &-1\\
0 & 1 & 0\\
-1&0&0
\end{array}\right)
\left(\begin{array}{ccc}
1 & 0 & 0\\
\alpha & 1 & 0\\
\beta &\overline \alpha&1
\end{array}\right)\\
&=\left(\begin{array}{ccc}
-\beta & -\overline \alpha &-1\\
\alpha & 1 & 0\\
-1&0&0
\end{array}\right).\notag
\)
\end{defi}

\begin{lemma}
\label{lemma:A}
In the Siegel projective model, we have 
\(\Frac\{\gamma_i\}_{i=1}^n =  A_{\gamma_1}\cdots A_{\gamma_n}(1:0:0).\)
\begin{proof}
Abstractly, we have the definition $\Frac\{\gamma_i\}_{i=1}^n = \iota \gamma_1 \iota \cdots \iota \gamma_n$. Using the identity element $0\in \Heis$, we may also write $\Frac\{\gamma_i\}_{i=1}^n = \iota \gamma_1 \iota \cdots \iota \gamma_n 0$. In the projective Siegel model, $0$ is interpreted as the point $(1:0:0)\in \CP^2$. The inversion $\iota$  and left multiplication by $\gamma_i$ are, respectively, interpreted as the unitary matrices $\mathbb U(\iota)$ and $\mathbb U(\gamma_i)$. Thus, $\Frac\{\gamma_i\}_{i=1}^n = A_{\gamma_1}\cdots A_{\gamma_n}(1:0:0)$,
as desired.
\end{proof}
\end{lemma}

We are now in position to characterize rational Heisenberg points in terms of their continued fraction expansion.
\begin{thm}
\label{thm:finiterational}
Let $h \in \Heis$. Then $h\in \Heis(\Q)$ if and only if $h = \Frac \{\gamma_i\}_{i=0}^n$ for some finite sequence $\{\gamma_i\}_{i=0}^n$.
\begin{proof}
Suppose $h = \Frac \{\gamma_i\}_{i=0}^n$. Then it is clear from the definition of $\Frac$ and the fact that $\gamma_i \in \Heis(\Z)$ that $h\in  \Heis(\Q)$.

Conversely, fix $K=K_D$ and assume by way of contradiction that there exists an element $h \in \Heis(\Q)$ with an infinite continued fraction sequence $CF(h) = \{\gamma_i\}_{i=1}^\infty$. Without loss of generality, we may assume $h\in K$ (this corresponds to discarding the digit $\gamma_0$ of $h$).  

The idea of the proof is to show that the forward iterates $h_i$ of $h$ can be written as fractions whose denominators decrease with $i$.  Write, in planar Siegel coordinates,
\( h=\left( \frac{r}{q}, \frac{p}{q}\right),\)
with $q,r,p\in\Z[\ii]$. Because $h\in K$, we have by Lemma \ref{lemma:gaugeSiegel} that $\norm{p/q}\leq \rad(K)^2 < 1$.

Consider the first forward iterate $h_1 = Th = \gamma_1^{-1}\iota h$ as a vector in $\C^3$:
\( 
\left(\begin{array}{c}q^{(1)}\\r^{(1)}\\p^{(1)}\end{array}\right):=
A_{\gamma_1}^{-1}\left(\begin{array}{c}q\\r\\p\end{array}\right) = 
\left(\begin{array}{ccc}
0&0&-1\\
0&1& \alpha_1\\
-1&-\overline{\alpha_1}&-\overline{\beta_1}
\end{array}\right)
\left(\begin{array}{c}q\\r\\p\end{array}\right)
=
\left(\begin{array}{c}-p\\r+\alpha_1 p\\-q-\overline{\alpha_1} r - \overline{\beta_1} p\end{array}\right)
\)
Thus, $h_1$ is a rational point with planar Siegel coordinates $h_1=\left( \frac{r^{(1)}}{q^{(1)}}, \frac{p^{(1)}}{q^{(1)}}\right)$. Furthermore, we have $q^{(1)}=-p$, so that
\[ \norm{\frac{q^{(1)}}{q}} = \norm{\frac{p}{q}} = \Norm{h}^2 < \rad(K)<1.\]

Repeating this procedure recursively, we have rational coordinates $h_i=\left( \frac{r^{(i)}}{q^{(i)}}, \frac{p^{(i)}}{q^{(i)}}\right)$ for each forward iterate $h_i$, satisfying $\norm{q^{(i)}} = \norm{p^{(i-1)}}$. Since $h_i\in K$ for each $i$, we obtain for each $n$:
\[\label{fla:growth} \norm{q^{(n)}} \leq \norm{q}(\rad(K))^{2n}\]
For sufficiently large $n$, we conclude $\norm{q^{(n)}}<1$, which implies that $q^{(n)}=0$, but that is only possible if $h_{n-1}=0$ and $CF(h)$ is, in fact, finite.
\end{proof}
\end{thm}

As a corollary to the proof of Theorem \ref{thm:finiterational}, we obtain
\begin{thm}[Denominator Growth Theorem]
\label{thm:denominator}
Let $h \in \Heis(\Q)$, with $CF(h)=\{\gamma_i\}_{i=0}^n$. Suppose one can write $h$ as a fraction with denominator $q\in\Z[\ii]$. Then,
\( \norm{q}\ge 2^{n/2}.\)
\begin{proof}
The result follows directly from (\ref{fla:growth}), using either fundamental domain $K$ in Lemma \ref{lemma:radii}, with radius bounded by $\sqrt[4]{1/2}$, and the fact that $q_n\neq 0$.
\end{proof}
\end{thm}

\begin{remark}
One may hope for a stronger statement that for a sequence $\{\gamma_i\}_{i=1}^\infty$ of elements of $\Heis(\Z)$, the norms of the denominators $q_n$ of the partial fractions $\Frac \{\gamma_i\}_{i=1}^n$ are an increasing sequence. However, we are unable to prove this without assuming that $\Norm{\gamma_i}\geq 2$ for all $i$. Indeed, the corresponding statement is false for some variants of continued fractions, see \S \ref{sec:classicalcf}.
\end{remark}

%%%%%%%%%%%%%%%%%%%%%%%%%%
\subsection{Recursive Formula}
%%%%%%%%%%%%%%%%%%%%%%%%%%
We will now find a simple recursive formula for $\Frac\{\gamma_i\}_{i=1}^\infty$. 

\begin{defi}
Let $\{\gamma_i\}$ be a sequence of elements of $\Heis[\Z]$. Define 
\(Q_n := A_{\gamma_1}\cdots A_{\gamma_n},\)
\((q_n, r_n, p_n) := Q_n(1,0,0).\) 
\end{defi}

We have the following by Lemma \ref{lemma:A}.

\begin{lemma}
\label{lemma:QFirstColumn}
In the above notation, $\Frac\{\gamma_i\}_{i=1}^n = \left(\frac{r_n}{q_n}, \frac{p_n}{q_n}\right)$, in the planar Siegel model.
\end{lemma}

\begin{remark}
It should be noted that Theorem \ref{thm:denominator} does not imply that $q_n \ge 2^{n/2}$.  Recall from Remark \ref{remark:twoexpansions} that if $CF(h)= \{\gamma_i\}_{i=0}^\infty$, we do not necessarily have that $CF(\Frac \{\gamma_i\}_{i=0}^n) = \{\gamma_i\}_{i=0}^n$.  
\end{remark}

Lemma \ref{lemma:QFirstColumn} states that the partial fraction $\Frac\{\gamma_i\}_{i=1}^n$ is encoded in the matrix $Q_n$. As in the case of regular continued fractions, $Q_n$ stores additional information:

\begin{lemma}
\label{lemma:Qn}
In the above notation, the matrices $Q_n$ have the form
\(Q_n=\left(\begin{array}{ccc}
q_n & \tilde q_n &-q_{n-1}\\
r_n & \tilde r_n & -r_{n-1}\\
p_n & \tilde p_n & -p_{n-1}
\end{array}\right),\)
where the elements $\tilde q_n, \tilde r_n, \tilde p_n$ are given by:
\(
&\tilde q_n = (-1)^n \overline{r_n q_{n-1}-q_n r_{n-1}},\\
&\tilde r_n = (-1)^n \overline{p_n q_{n-1} - q_n p_{n-1}},\\
&\tilde p_n = (-1)^n \overline{p_n r_{n-1}-r_n p_{n-1}}.
\)
Moreover, the matrix $Q_n$ has determinant $(-1)^n$.
\begin{proof}
The first column of $Q_n$ is as stated by the definition of the vector $(q_n, r_n, p_n)$. The third column follows from the identity $Q_n = Q_{n-1}A_{\gamma_n}$. The determinant follows from the fact that each $A_{\gamma_i}$ has determinant $-1$.  Finally, the second column follows from the ``cross product'' Lemma \ref{lemma:crossproduct}.
\end{proof}
\end{lemma}

We record the following for later use:
\begin{lemma}
\label{lemma:qn2}
The identity $Q_n^\dagger \J = \J Q_n^{-1}$  is equivalent to
\begin{align*}
&\overline{\left(\begin{array}{ccc}
-p_n &r_n &-q_n\\
-\tilde p_n &\tilde r_n &-\tilde q_n\\
p_{n-1} & -r_{n-1} & q_{n-1}
\end{array}\right)}\\
&\qquad=
(-1)^n \left(\begin{array}{ccc}
p_n\tilde r_n-\tilde p_n r_n &\tilde p_n q_n -\tilde q_n p_n  &r_n \tilde q_n - \tilde r_n q_n\\
p_{n-1}r_n-p_nr_{n-1} &p_nq_{n-1}-p_{n-1}q_n &r_{n-1}q_n-r_n q_{n-1}\\
p_{n-1}\tilde r_n -\tilde p_n r_{n-1} &\tilde p_n q_{n-1}-p_{n-1}\tilde q_n &r_{n-1}\tilde q_n-\tilde r_nq_{n-1}
\end{array}\right).
\end{align*}
\end{lemma}

We can now obtain a recursive form for the partial fractions $\Frac\{\gamma_i\}_{i=1}^n$.
\begin{thm}
\label{thm:recursive}
Let $\{\gamma_i\}_{i=1}^\infty$ be a sequence of elements of $\Heis(\Z)$ represented in the planar Siegel model by the vectors $\{(\alpha_i, \beta_i)\}_{i=1}^\infty$. Let $(q_{-1}, p_{-1}, r_{-1}) = (0,0,1)$ and $(q_0,p_0,r_0)=(1,0,0)$. Define, recursively, for $n\geq 0$,
\(
\left(\begin{array}{c}q_{n+1}\\r_{n+1}\\p_{n+1}\end{array}\right)=
\left(\begin{array}{ccc}
q_n & (-1)^n\overline{r_n q_{n-1}-q_n r_{n-1}} &-q_{n-1}\\
r_n & (-1)^n \overline{p_n q_{n-1} - q_n p_{n-1}} & -r_{n-1}\\
p_n & (-1)^n \overline{p_n r_{n-1}-r_n p_{n-1}} & -p_{n-1}
\end{array}\right)
\left(\begin{array}{c}
-\beta_{n+1}\\
\alpha_{n+1}\\
-1
\end{array}
\right)
\)
Then for each $n$ we have, in the planar Siegel model, 
\(\Frac \{\gamma_i\}_{i=1}^n = \left( \frac{r_n}{q_n}, \frac{p_n}{q_n}\right).\)
\begin{proof}%[Proof of Theorem \ref{thm:recursive}]
Earlier in the section, we defined matrices $A_{\gamma_i}$ (which append the digit $\gamma_i$ to a continued fraction) and $Q_n = A_{\gamma_1}\cdots A_{\gamma_n}$. We set $(q_n, r_n, p_n)=Q_n (1,0,0)$. We claim that this agrees with the definition in the statement of the theorem. Lemma \ref{lemma:QFirstColumn} will then tell us that 
$\Frac \{\gamma_i\}_{i=1}^n = \left( \frac{r_n}{q_n}, \frac{p_n}{q_n}\right)$.

Taking $Q_0$ to be the identity matrix, the following computation provides the equivalence (see the definition of $A_{\gamma_{n+1}}$ and Lemma \ref{lemma:Qn} for the form of the two matrices).
\(
\left(\begin{array}{c}q_{n+1}\\r_{n+1}\\p_{n+1}\end{array}\right)
&=Q_{n+1}
\left(\begin{array}{c}1\\0\\0\end{array}\right)\\
&=
Q_n A_{\gamma_{n+1}}\left(\begin{array}{c}1\\0\\0\end{array}\right)
\\
& = 
\left(\begin{array}{ccc}
q_n & \tilde q_n &-q_{n-1}\\
r_n & \tilde p_n & -r_{n-1}\\
p_n & \tilde r_n & -p_{n-1}
\end{array}\right)
\left(\begin{array}{ccc}
-\beta_{n+1} & -\overline \alpha_{n+1} &-1\\
\alpha_{n+1} & 1 & 0\\
-1&0&0
\end{array}
\right)
\left(\begin{array}{c}1\\0\\0\end{array}\right)\\
\\
& = 
\left(\begin{array}{ccc}
q_n & \tilde q_n &-q_{n-1}\\
r_n & \tilde p_n & -r_{n-1}\\
p_n & \tilde r_n & -p_{n-1}
\end{array}\right)
\left(\begin{array}{c}
-\beta_{n+1}\\
\alpha_{n+1}\\
-1
\end{array}
\right)
\)
Rewriting $\tilde q_n, \tilde r_n, \tilde p_n$ in terms of the other terms in $Q_n$ completes the proof. 
\end{proof}
\end{thm}

%%%%%%%%%%%%%%%%%%%%%%%%%%%%%%%%%%%%%%%%
\subsection{Continued Fraction Representation Theorem}
%%%%%%%%%%%%%%%%%%%%%%%%%%%%%%%%%%%%%%%%
We are now ready to prove the convergence of continued fraction expansions. In fact, we obtain a variation on the strong convergence property. While we do not obtain strong convergence in the sense of Schweiger \cite{schweiger}, our convergence estimate is obtained via a similar method to strong convergence for regular continued fractions.  We hope to improve this estimate and explore applications to Diophantine approximation in an upcoming paper.

We also note that we obtain such an explicit convergence estimate by exploiting  a special form for  $Q_n^{-1}$ that follows from the identity $M^\dagger \J M = \J$ that defines $U(2,1)$.  Other continued fraction theories are complicated by the lack of a simple form for $Q_n^{-1}$.

Before we can prove convergence, we need to show that $q_n$ will never equal $0$.  We prove this in two steps.

\begin{lemma}\label{lemma:fracqn}
We have
\[\label{frac:qn}
\left(\begin{array}{c}
q_{n}+\tilde q_{n}u_n - q_{n-1}v_n\\
r_{n}+\tilde r_{n}u_n - r_{n-1}v_n\\
p_{n}+\tilde p_{n}u_n - p_{n-1}v_n
\end{array}\right)
= 
(-1)^{n}
\left(\begin{array}{c}
\frac{1}{v v_1 \cdots v_{n-1}}\\
\frac{u}{v v_1 \cdots v_{n-1}}\\
\frac{1}{v_1 \cdots v_{n-2}}
\end{array}\right)
\]

\begin{proof}
By Lemma \ref{lemma:Qn}, the vector on the left-hand side of \eqref{frac:qn} equals
\(
 Q_n(1, u_n, v_n)= A_{\gamma_1} \cdots A_{\gamma_n} (1, u_n, v_n).
\)
Recall that the forward iterates of $h$ are given by $h_i=T^ih=A_{\gamma_i}^{-1} \cdots A_{\gamma_1}^{-1}h$, and have planar Siegel coordinates $(u_i, v_i)$, corresponding to the points $(1:u_i:v_i)\in \CP^2$.

More generally, we have $A_{\gamma_i} \cdots A_{\gamma_n} (1: u_n: v_n)=h_i$. Write $A_{\gamma_n} (1, u_n, v_n) =: (a,b,c)$. Since $A_{\gamma_n}$ has the form (see Definition \ref{defi:Agamma})
\(\left(\begin{array}{ccc}
-\beta & -\overline \alpha &-1\\
\alpha & 1 & 0\\
-1&0&0
\end{array}\right),\)
we have that $c=-1$. Since $(b/a, c/a)=(u_{n-1}, v_{n-1})$, we conclude
\(A_{\gamma_n}(1,u_n, v_n) = \left(-\frac{1}{v_{n-1}},-\frac{u_{n-1}}{v_{n-1}}, -1\right).\)
Continuing in the same fashion we see that
\(A_{\gamma_{n-1}}A_{\gamma_n}(1,u_n, v_n) &= A_{\gamma_{n-1}}\left(-\frac{1}{v_{n-1}},-\frac{u_{n-1}}{v_{n-1}}, -1\right)\\
                                                                       &= \left(\frac{1}{v_{n-1}v_{n-2}},\frac{u_{n-2}}{v_{n-1}v_{n-2}}, \frac{1}{v_{n-1}}\right).
\)
After $n$ iterations, the process yields the desired formula.
\end{proof}
\end{lemma}

\begin{lemma}\label{lemma:nonzero}
For $n\ge 0$, we have that $q_n$ never equals $0$.
\end{lemma}

\begin{proof}
Assume, by way of contradiction, that $q_n=0$.  Then by Lemmas \ref{lemma:Qn} and \ref{lemma:qn2}, we have that $\tilde q_n =0$ as well ($r_n$ also must equal $0$, but we will not use this fact).  Since the matrix $Q_n$ has determinant $(-1)^n$ and each entry is a Gaussian integer,  $q_{n-1}$ must have norm $1$.

Therefore, we have that 
\[
\left|q_{n}+\tilde q_{n}u_n - q_{n-1}v_n\right| = |v_n| <1.
\]
However by Lemma \ref{lemma:fracqn}, we have that
\[
\left|q_{n}+\tilde q_{n}u_n - q_{n-1}v_n\right| = |v v_1 v_2 \dots v_{n-1}|^{-1} >1,
\]
which is a contradiction.  Therefore our assumption that $q_n=0$ must be false.
\end{proof}

Now we can continue with the proof of convergence.

\begin{thm}
\label{thm:conv}
Let $h \in \Heis$ and let $K$ be a fundamental domain for $\Heis(\Z)$ with $\rad(K)<1$. Then 
\(\Frac CF(h)=h.\)
Furthermore, if $CF(h)=\{\gamma_i\}$ is a sequence with at least $n$ terms, then the rational approximants satisfy
\(d\left(\Frac \{\gamma_i\}_{i=0}^n, h\right) \leq \rad(K)^{n+1}
\)
for both rational and irrational points in $\Heis$. Let $q_n$ be the denominator of the $n^{th}$ rational approximate. Then we in fact have
\(d\left(\Frac \{\gamma_i\}_{i=0}^n, h\right) \leq \frac{\rad(K)^{n+1}}{\norm{q_n}^{1/2}}.
\)

\begin{proof}
Recall from Lemma \ref{lemma:QFirstColumn}
 that the associated rational approximates $\Frac\{\gamma_i\}_{i=1}^n$ have planar Siegel coordinates $\left(\frac{r_n}{q_n}, \frac{p_n}{q_n}\right)$, associated to the vector $(q_n, r_n, p_n)\in \C^3$. Recall also that the forward iterates $T^nh$ have planar Siegel coordinates $(u_n, v_n)$, and we have $\norm{v_n}^{1/2} \leq \rad(K)<1$.

To prove the thoerem, it suffices to show that
\(d\left(\Frac \{\gamma_i\}_{i=0}^n, h\right) = \frac{\prod_{i=0}^n \norm{v_i}^{1/2}}{\norm{q_n}^{1/2}}.\)
Indeed, by Lemmas \ref{lemma:siegelMinus} and \ref{lemma:gaugeSiegel}, we have
\(
d\left(\Frac \{\gamma_i\}_{i=0}^n, h\right) &= d\left(\left(\frac{r_n}{q_n}, \frac{p_n}{q_n}\right), h\right)\\
      &= \Norm{\left( u-\frac{r_n}{q_n}, v- \overline{\left(\frac{r_n}{q_n}\right)}u+\overline{\left(\frac{p_n}{q_n}\right)}\right)}\\
      &= \norm{v- \overline{\left(\frac{r_n}{q_n}\right)}u+\overline{\left(\frac{p_n}{q_n}\right)}}^{1/2}\\
      &= \frac{\norm{\overline{q_n}v- \overline{{r_n}}u+\overline{{p_n}}}^{1/2},}{\norm{q_n}^{1/2}}.
\)

We now view $h$ as the vector $(1, u,v)$ and represent the operation $T^n$ by the unitary matrix $Q^{-1}_n$. The vector
\(
Q_n^{-1} \left(  \begin{array}{c} 1 \\ u \\ v \end{array} \right) = \overline{
\left(\begin{array}{ccc}
-p_{n-1} & r_{n-1} & - q_{n-1} \\
-\tilde p_n & \tilde r_n & -\tilde q_n \\
p_n & -r_n & q_n
\end{array}\right)}  \left(  \begin{array}{c} 1 \\ u \\ v \end{array} \right)
\)
is then a scalar multiple of $(1, u_{n}, v_{n})$.  In particular,
\(
v_{n}= - \frac{\overline{p_n}-\overline{r_n}u+\overline{q_n} v}{\overline{p_{n-1}}-\overline{r_{n-1}}u+\overline{q_{n-1}} v}.
\)

By multiplying this formula together for various indices we obtain
\begin{align*}
\prod_{i=1}^n v_i &= (-1)^n \prod_{i=1}^n  \frac{\overline{p_i}-\overline{r_i}u+\overline{q_i} v}{\overline{p_{i-1}}-\overline{r_{i-1}}u+\overline{q_{i-1}} v}\\
&= (-1)^n  \frac{\overline{p_n}-\overline{r_n}u+\overline{q_n} v}{\overline{p_{0}}-\overline{r_{0}}u+\overline{q_{0}} v}\\
&= (-1)^n \frac{\overline{p_n}-\overline{r_n}u+\overline{q_n} v}{v}
\end{align*}
This yields the interesting formula
\[\label{fla:cool}
\overline{p_n}-\overline{r_n}u+\overline{q_n} v= (-1)^n \prod_{i=0}^n v_i.
\]
We then have
\(
d\left(\Frac \{\gamma_i\}_{i=0}^n, h\right)       &= \frac{\norm{\overline{q_n}v- \overline{{r_n}}u+\overline{{p_n}}}^{1/2},}{\norm{q_n}^{1/2}}\\
 &= \frac{\norm{\prod_{i=0}^n v_i}^{1/2}}{\norm{q_n}^{1/2}}.
\)
Noting that $q_n\in \Z[i]$ and that $q_n\neq 0$ by Lemma \ref{lemma:nonzero} completes the proof.
\end{proof}
\end{thm}

\begin{cor}\label{cor:qn}
If $h\in K \setminus \Heis(\Q)$, then $|q_n|$ tends to $\infty$.
\end{cor}

\begin{proof}
This follows almost immediately from the fact that there are only finitely many rational points $(\frac{r}{q},\frac{p}{q})\in \Sieg$ that are written lowest terms, are inside the unit sphere, and have $|q|$ bounded.  Since the volume of $\epsilon$-radius balls centered at these points shrinks to zero as $\epsilon$ shrinks to zero, no irrational point $h$ can be arbitrarily well approximated by such points.
\end{proof}

As a corollary to the proof of Theorem \ref{thm:conv} we obtain a new form of the classical formula for regular continued fractions:
\(
\left| x-\frac{p_n}{q_n} \right| =\frac{1}{q_n(q_{n+1}+q_n\cdot T^{n+1} x )}.
\)
The left-hand side of this formula may be considered to be the distance between $x$ and the point $p_n/q_n$.  Recall that in Theorem \ref{thm:conv} we showed that
\[
d\left(\Frac \{\gamma_i\}_{i=0}^n, h\right)  =  \norm{v- \overline{\left(\frac{r_n}{q_n}\right)}u+\overline{\left(\frac{p_n}{q_n}\right)}}^{1/2}.
\]

\begin{thm}
\label{thm:classicalformula}
Let $h\in \Heis$ with continued fraction digits $CF(h)=\{\gamma_i\}$, associated to a fundamental domain $K$ with $\rad(K)<1$, and rational approximates $\Frac \{\gamma_i\}_{i=1}^n = \left(\frac{r_n}{q_n}, \frac{p_n}{q_n}\right)$. Then, in the notation of Lemma \ref{lemma:Qn},
\(
 v - \overline{ \left( \frac{r_n}{q_n} \right) } u  +\overline{\left( \frac{p_n}{q_n} \right)} = \frac{1}{ \overline{q_n} (q_{n+1}+\tilde q_{n+1}u_{n+1} - q_{n}v_{n+1})}.
\)
\end{thm}

\begin{proof}
This follows immediately from \eqref{fla:cool} and Lemma \ref{lemma:fracqn}.
\end{proof}

%%%%%%%%%%%%%
\subsection{Uniform convergence}
%%%%%%%%%%%%%%
We continue with the assumptions of Theorem \ref{thm:conv} and the notation of Lemma \ref{lemma:Qn}. The purpose of this section is to study the points $\left(\frac{\tilde r_n}{\tilde q_n}, \frac{\tilde p_n}{\tilde q_n}\right)$, and to understand when they converge (in the appropriate sense) to $h$.

We will say a point $h=(u,v)$ is degenerate if $u_n=0$ for some $n$, and non-degenerate otherwise. Degenerate points are named such since the dynamical properties of such points eventually simplify into those of one-dimensional real continued fractions. We will prove the following theorem in this section.

\begin{thm}\label{thm:uniformity}
Let $h\in K$. If $h$ is non-degenerate, then $\left(\frac{\tilde r_n}{\tilde q_n}, \frac{\tilde p_n}{\tilde q_n}\right)$ converges to $h$ (in the Euclidean sense as elements of $\mathbb{C}^2$) as $n$ tends to infinity. If $h$ is degenerate, then the points $\left(\frac{\tilde r_n}{\tilde q_n}, \frac{\tilde p_n}{\tilde q_n}\right)$ are eventually cosntant.
\end{thm}

It should be emphasized that that \emph{none} of the points $\left(\frac{\tilde r_n}{\tilde q_n}, \frac{\tilde p_n}{\tilde q_n}\right)$ are actually in $\Sieg$, due to the following lemma.

\begin{lemma}\label{lemma:tildenotsiegel}
 We have
\[\label{eq:determinantresult}
 |\tilde r_n|^2 - 2 \Re(\overline{\tilde q_n} \tilde p_n) =1.
\]
\end{lemma}

\begin{proof}
 This can be easily found by using the fact from Lemma \ref{lemma:Qn} that $\operatorname{det}(Q_n)=(-1)^n$. If we write down this determinant in terms of the matrix coefficients and then simplify, this gives the left-hand side of \eqref{eq:determinantresult} times a factor of $(-1)^n$.
\end{proof}

The importance of non-degeneracy comes from the following Lemma.

\begin{lemma}\label{lemma:nonzerotildeqn}
If $u_n=0$, then $u_{n+1}=0$ and $\tilde q_{n+1}=\tilde q_n$.  

If $h$ is non-degenerate, then $|\tilde q_n|$ tends to infinity as $n$ grows.
\end{lemma}

Note that it is possible for $\tilde q_n$ to equal $0$, but that if $h$ is non-degenerate then this can only occur finitely many times.

\begin{proof}[Proof of Lemma \ref{lemma:nonzerotildeqn}]
If $u_n=0$, then the corresponding point $T^n(z,t)$ has $z$-coordinate equal to $0$.  A quick calculation shows that  $\gamma_{n+1}=[\iota T^n(z,t)]$ must  have $z$-coordinate equal to $0$, and therefore, so must $T^{n+1}(z,t)$.  Converting this back to $(u,v)$-coordinates shows that $u_{n+1}=0$, and since the matrix $A_{\gamma_{n+1}}$ takes the form
\(
\left( \begin{array}{ccc} * & 0 & -1\\
0 & 1 & 0\\
-1 & 0 & 0\\
\end{array} \right),
\)
we have that $\tilde q_{n+1} = \tilde q_n$.

Now suppose $h$ is non-degenerate.  In particular assume that if $n>N$, then $u_n \neq 0$. By modifying the argument that yielded \eqref{fla:cool}, we can easily obtain
\[\label{fla:cool2}
 \overline{\tilde p_n} - \overline{\tilde r_n} u + \overline{\tilde q_n} v = (-1)^{n-1} \prod_{i=0}^{n-1} v_i \cdot u_n.
\]
Since each $u_i$ and $v_i$ has norm strictly between $0$ and $1$, we have that the right-hand side of \eqref{fla:cool2} comes arbitrarily close to, but never equals, $0$ as $n$ increases.

Suppose, by way of contradiction, that there exist infinitely many $\{n_m\}_{m=1}^\infty$ such that $|\tilde q_{n_m}| < M$ for some $M$. Lemma \ref{lemma:qn2} implies that
\(
 \overline{q_n} &=(-1)^{n+1} \left( r_n \tilde q_n - \tilde r_n q_n\right)\\
\overline{r_n} &=(-1)^n  \left( \tilde p_n q_n - \tilde q_n p_n\right)
\)
and therefore
\[
\tilde r_n &= \frac{r_n}{q_n} \tilde q_n + (-1)^n \frac{\overline{q_n}}{q_n} \label{eq:tildern}\\
\tilde p_n &= \frac{p_n}{q_n} \tilde q_n + (-1)^n \frac{\overline{r_n}}{q_n}. \label{eq:tildepn}
\]
Since there are only finitely many values that $\tilde q_{n_m}$ can take, these equations imply that there are also only finitely many values that the tuple $(\tilde q_{n_m}, \tilde r_{n_m}, \tilde p_{n_m})$ can take; and hence only finitely many values that $\overline{\tilde p_{n_m}} - \overline{\tilde r_{n_m}} u + \overline{\tilde q_{n_m}} v$ can take.  This contradicts the fact that the left-hand side of \eqref{fla:cool2} gets arbitrarily close to $0$ without equaling it.

Hence $|q_n|$ must tend to infinity as $n$ grows.
\end{proof}

Note that \eqref{fla:cool2} provides a necessary condition for non-degeneracy: if there exist do not exist $a,b,c \in \Z[\ii]$ with $a+bu+cv = 0$ and $|b|^2-2 \operatorname{Re}(\overline{c} a) = 1$, then $h$ is non-degenerate. It is not clear whether this is a sufficient condition as well.

\begin{proof}[Proof of Theorem \ref{thm:uniformity}]
Assume that $h$ is non-degenerate. From \eqref{eq:tildern} and \eqref{eq:tildepn}, we have 
\(
\frac{\tilde r_n}{\tilde q_n} &= \frac{r_n}{q_n} +  (-1)^n \frac{\overline{q_n}}{q_n}\cdot \frac{1}{\tilde q_n} \\
 \frac{\tilde p_n}{\tilde q_n} &= \frac{p_n}{q_n} + (-1)^n \frac{\overline{r_n}}{q_n }\cdot \frac{1}{\tilde q_n} ,
\)
provided $n$ is large enough so that $\tilde q_n$ is non-zero.  We have that $r_n/q_n$ and $p_n/q_n$ converge to $u$ and $v$ respectively.  Since both $\overline{q_n} / q_n$ and $\overline{r_n}/q_n$ are bounded, and since $|\tilde q_n|$ goes to infinity, this proves that $\tilde r_n / \tilde q_n$ and $\tilde p_n/\tilde q_n$ converge to $u$ and $v$ respectively.
\end{proof}

%%%%%%%%%%%%%
\section{Dynamical Properties}
%%%%%%%%%%%%%%%
\label{sec:gauss}

Let $K$ be a fundamental domain for $\Heis(\Z)$ with $\rad(K)<1$ and $T:K \rightarrow K$ the associated Gauss map $T(h) = \floor{\iota h}^{-1}\iota h$. From the dynamical systems perspective, the following questions are immediate:

\begin{question}
\label{q:invariant}
Does there exist a $T$-invariant measure $\mu$ on $K$ such that $\mu$ is absolutely continuous with respect to Lebesgue measure? 
\end{question}
\begin{question}
\label{q:ergodic}
Does there exist $\mu$ satisfying Question \ref{q:invariant} such that  $T$ is ergodic with respect to $\mu$?
\end{question}

We first review some results in ergodic theory in Section \ref{sec:dynamicaltools}, in the context of the \emph{complex} Gauss map. In Section \ref{sec:baseb} we show that for a simpler dynamical system associated to base-$b$ expansions in the Heisenberg group, Questions \ref{q:invariant} and \ref{q:ergodic} can be answered affirmatively.

\subsection{Dynamical Tools}
\label{sec:dynamicaltools}
The main difficulty in answering Questions \ref{q:invariant} and \ref{q:ergodic} comes from the complicated combinatorics of the Heisenberg Gauss map. Indeed, similar issues arize when one considers continued fractions over $\C$. 

\begin{defi}
Let $K=[-1/2,1/2)\times [-1/2,1/2) \subset \C$. The \emph{complex Gauss map} $T: K \rightarrow K$ is given by
\[T(z) = \iota(z) - \class{\iota(z)}.\]
where $\iota(z) = \frac{1}{\overline{z}}$, and $\class{z}$ takes $z$ to the nearest Gaussian integer (with the choice $\class{K}=0$).
\end{defi}

The complex Gauss map is continuous on \emph{cylinder sets} $C_\gamma = \iota(K+\gamma)\cap K$, for each $\gamma\in \Z[\ii]$. For sufficiently large $\gamma$, $\iota(K+\gamma)$ lies inside $K$, so that one has $T(C_\gamma)=K$. One says that such cylinders are \emph{full}. As Figure \ref{fig:complexGauss} demonstrates, some cylinders $C_\gamma$ are not full; in other words, the dynamical system \emph{non-Markov}.

\begin{figure}[h!]
\label{fig:complexGauss}
\includegraphics[width=.45\textwidth]{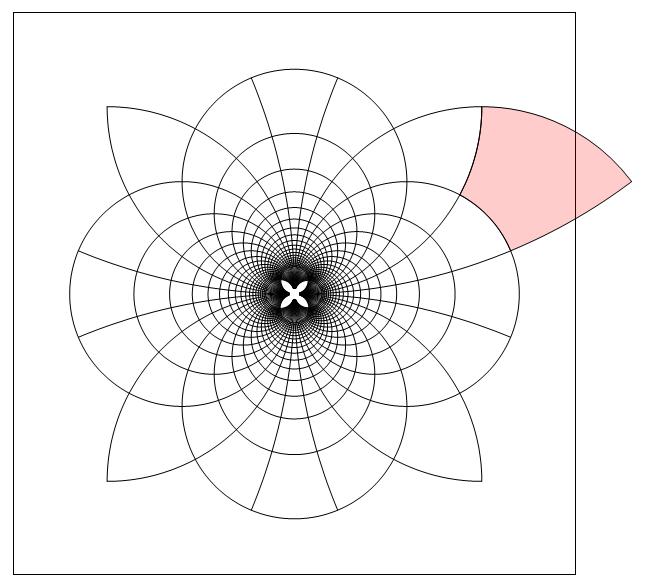}
\caption{Cylinder structure of the complex Gauss map.}
\end{figure}

\begin{defi}[Fibered System]
\label{defi:fiberedsystem}
More generally, consider a topological space $K$ and a piecewise-continuous mapping $T:K \rightarrow K$. Let $\mathcal D$ be a countable \emph{digit set}, and assume that $T$ is continuous and invertible on sets $C_{\{w_1\}}\subset K$, for various $w_1 \in D$. As with continued fractions, one associates with each $h\in K$ the sequence $\{w_1, \ldots\}$ of 
digits satisfying $T^{n-1}h \in C_{w_n}$. A sequence arising in this way (or any of its initial subsequences) is called \emph{admissible}. To each finite admissible sequence $\{w_1, \ldots, w_n\}$, one associates the cylinder set $C_{\{w_1, \ldots, w_n\}}$, consisting of the points in $K$ whose digit sequence starts with $\{w_1, \ldots, w_n\}$. The collection of cylinders is known as a \emph{fibered system}.
\end{defi}

\begin{thm}[See Theorems 4 and 8 in \cite{schweiger}] 
\label{thm:schweiger}
Let $T$ give rise to a fibred system over a set $K$, with digit set $\mathcal{D}$.  Let $\lambda$ be some measure on $K$.
Suppose
\begin{enumerate}
\item $\lambda(K) = 1$;
\item The system is Markov (that is, all the cylinders are full);
\item For any infinite admissible sequence $w=\{w_1,w_2,\dots\}$ of digits from $\mathcal{D}$, we have
\(
\lim_{n\to \infty} \diam C_{\{w_1,w_2,\dots,w_n\}} = 0;
\)
\item There is a constant $C\ge 1$ such that for all finite admissible strings $w$ of length $n$,
\(
\frac{\sup_{y\in T^n C_w} J_yT^n}{\inf_{y\in T^n C_w}J_yT^n}
 \le C.
\)
\end{enumerate}
Then $T$ is ergodic and admits a unique finite invariant measure $\mu$ absolutely continuous with respect to $\lambda$ (furthermore, $\mu$ is equivalent to $\lambda$).
\end{thm}

For a non-Markov system satisfying the remaining conditions of Theorem \ref{thm:schweiger}, one can recover ergodicity by answering the following question affirmatively for almost every point in $K$:

\begin{question}
\label{q:fullcylinder}
Let $h\in K$ have digit sequence $\{w_i\}_{i=1}^\infty$. Does there exist $n(h)\in \N$ such that the cylinder $C_{\{w_1, \ldots, w_{n(h)}\}}$ is full?
\end{question}

If Question \ref{q:fullcylinder} can be answered affirmatively for almost every point, one defines an auxilliary ``speedup'' mapping $h\mapsto T^{n(h)}h$. The associated system is Markov, and one recovers the results of Theorem \ref{thm:schweiger} for both the speedup map and the original mapping $T$.

This approach may be useful for the Heisenberg Gauss map. Indeed, while the associated system is not Markov, the other conditions are immediate from Theorem \ref{thm:conv} and the following lemma (note that the exponent is $4$, not $3$ as one might expect).

\begin{lemma}Let $K$ be a fundamental domain for $\Heis(\Z)$ with $\rad(K)<1$. Then for almost every $h\in K$, the Jacobian determinant of the $n^{th}$ power of the Heisenberg Gauss map $T^n$ is given by 
\(J_hT^n = \prod_{i=0}^{n-1} \Norm{h_i}^4,\)
where $h_i = T^ih$ for $i=0, \ldots, n$.
\begin{proof}
Using the CF digits of $h$, we may write $T^n = \gamma_n^{-1} \iota \cdots \gamma_1^{-1} \iota$ near $h$ (unless $h$ is at the boundary of a cylinder). The left translations by $\gamma_i$ are shears and have Jacobian $1$. The Jacobian of the inversion $\iota$ at an intermediate point $h_i$ is given by Lemma \ref{lemma:Jacobianiota} as $\Norm{h_i}^4$. The lemma follows from the chain rule.
\end{proof}
\end{lemma}

As far as we know, Question \ref{q:fullcylinder} remains open for both the complex and Heisenberg Gauss maps, with respect to any fundamental domain $K$. However, in the case of nearest-integer complex Gauss map, the combinatorics of the system are sufficiently tractable to prove ergodicity with respect to a measure equivalent to Lebesgue measure, see \cite{Hensley}.

Lastly, we remark that some results are available for non-Markov systems, see, e.g., \cite{saussol}, but applying each theory requires a detailed understanding of the combinatorics of the dynamical system.

\subsection{Base-$b$ expansions}
\label{sec:baseb}
We now focus on a more tractable dynamical system that is a direct generalization of base-$b$ expansions.

Fix a fundamental domain $K$ for $\Heis(\Z)$ (we assume $K$ contains $0$), and an integer $b>1$. Define the base-$b$ mapping $T_b: K \rightarrow K$ by 
\[T_b(h) = \class{\delta_bh}^{-1}\delta_b h,\]
where $[\cdot]\in \Heis(\Z)$ is defined with respect to $K$ and $\delta_b$ is a metric dilation by factor $b$, see \S \ref{sec:geometric}. The fibered system associated to $T_b$ (see \ref{defi:fiberedsystem}) associates to each $h\in K$ a sequence of \emph{base-$b$ digits} $\gamma_i$ and forward iterates $h_i$:

\(&\gamma_0 = [h] &&h_0=\gamma_0^{-1}h,\\
&\gamma_{i+1}=[\delta_bh_i] &&h_{i+1}=\gamma_{i+1}^{-1}\delta_b h_i = T^{i+1}h.\)

\begin{thm}
Let $h\in \Heis, b\geq2$, and $\{\gamma_i\}\subset \Heis(\Z)$ the base-$b$ digits of $h$. Then one has
\(h =\lim_{n\rightarrow\infty}  \gamma_0 * \delta_{b^{-1}}(\gamma_1) *\cdots* \delta_{b^{-n}}(\gamma_n). \)
\begin{proof}
The convergence is clear, since $\delta_{b^{-1}}$ is a contraction by factor $b$, and furthermore a group isomorphism of $\Heis$ satisfying $\delta_{r_1} \circ \delta_{r_2} = \delta_{r_1 r_2}$.
\end{proof}
\end{thm}

\begin{thm}
\label{thm:b-ergodic}
The mapping $T_b$ is ergodic with respect to Lebesgue measure $\lambda$. Furthermore, $\lambda$ is the unique measure absolutely continuous with respect to Lebesgue measure for which $T_b$ is ergodic.
\end{thm}

For a generic fundamental domain $K$, the combinatorics of $T_b$ can be complicated; in particular, the associated cylinder sets are usually not full. This happens, for example, for both the cube $K_C$ and Dirichlet domain $K_D$ (indeed, the same is true in the complex plane). However, Strichartz showed in \cite{strichartz92} that there exists a fundamental domain $K_S$ for $\Heis(\Z)$ (depending on $b$) such that $\delta_b K_S$ decomposes as the disjoint union of integer translates of $K_S$ (in particular, the associated fibered system is Markov).

The idea of the proof is to observe that the dynamical system is simple for $K=K_S$. To see the same dynamical properties for a generic $K$, we project the dynamical system to the quotient \emph{nilmanifold} $\Heis(\Z)\backslash \Heis$.

\begin{remark}
Strichartz constructs $K_S$ for $b=2$ in a larger class of spaces. However, the results hold for arbitrary integers $b\geq 2$.
\end{remark}

\begin{proof}[Proof of Theorem \ref{thm:b-ergodic}]
Let $\pi: \Heis \rightarrow \Heis(\Z)\backslash \Heis$ be the natural projection map form the Heisenberg group to its \emph{nilmanifold} quotient. The mapping is a covering map, and since the action of $\Heis(\Z)$ on $\Heis$ is by measure-preserving isometries, $\pi$ induces a metric and measure on $\Heis(\Z)\backslash \Heis$ (we continue to refer to $\pi_* \lambda$ as Lebesgue measure). Furthermore, $\delta_b$ is a group homomorphism of $\Heis(\Z)$, so one has a well-defined mapping $\pi \delta_b \pi^{-1}: \Heis(\Z)\backslash \Heis \rightarrow \Heis(\Z)\backslash \Heis$. Likewise, $\pi T_b \pi^{-1}$ is well defined and, indeed, one has $\pi T_b \pi^{-1} = \pi \delta_b \pi^{-1}$. Because the restriction $\pi\vert_K: K \rightarrow \Heis(\Z)\backslash \Heis$ is measure-preserving bijection, it suffices to prove the theorem for the mapping $\pi \delta_b \pi^{-1}$ on $\Heis(\Z)\backslash \Heis$.

We remark that $\Heis(\Z)\backslash \Heis$ is analogous to the 3-torus $\Z^3\backslash \R^3$, but it is not homeomorphic to the 3-torus.

It is clear that the conditions of Theorem \ref{thm:schweiger} are satisfied for any choice of cylinders in $\Heis(\Z)\backslash \Heis$, except perhaps for the Markov property. We will obtain a Markov partition from the Strichartz tile $K_S$.

The decomposition of the dilated Strichartz tile $\delta_b K_S = \cup_{i=1}^{b^4} \gamma_i K_S$ provided by \cite{strichartz92} induces a decomposition $K_S = \cup_{i=1}^{b^4}\delta_{r^{-1}}\gamma_i K_S$, and therefore a decomposition of $\Heis(\Z)\backslash \Heis$ into cylinder sets. It is clear that these cylinders are full, so that Theorem \ref{thm:schweiger} applies.

We conclude that there exists a unique $\pi T_b \pi^{-1}$-invariant measure $\mu$ on $\Heis(\Z)\backslash \Heis$ that is absolutely continuous with respect to Lebesgue measure. On the other hand, it is easy to see that Lebesgue measure is indeed preserved by $\pi T_b \pi^{-1}$, so that $\mu$ is Lebesgue measure. Because $\pi$ conjugates the dynamical systems $(K, T_b)$ and $(\Heis(\Z)\backslash \Heis, \pi T_b \pi^{-1})$, we obtain the desired properties of $T_b: K \rightarrow K$.
\end{proof}

\begin{remark}
In the spirit of the proof of Theorem \ref{thm:b-ergodic}, one could view the Heisenberg Gauss map as a mapping of the nilmanifold $\Heis(\Z)\backslash \Heis$. However, unlike in the case of base-$b$ expansions, the induced mapping will depend on the particular choice of fundamental domain $K$, so a reduction of the ergodicity question to a dynamical system on a particularly nice fundamental domain is not possible.
\end{remark}

\section*{Acknowledgements}
The authors would like to thank Jayadev Athreya, Florin Boca, Jeremy Tyson, and  Matthew Stover for interesting conversations.

\bibliographystyle{amsplain}
\bibliography{bib}

\end{document}